%% file: main.tex
\documentclass[10pt]{article}
\usepackage[utf8]{inputenc}

\usepackage[a4paper,left = 1in,right = 1in, top = 1.5in, bottom = 1.5in]{geometry}
\usepackage{amsmath}
\usepackage{amsthm}
\usepackage{bm}
\usepackage{cancel}
\usepackage{amssymb}
\usepackage{dsfont}
\usepackage{commath}
\usepackage{mathtools}
\usepackage{float}
\usepackage{algorithm}

\usepackage[numbers]{natbib}

\input{Math_Operators}

\usepackage{subcaption}
\usepackage{graphicx}
\usepackage{color}

\usepackage{xcolor}
\usepackage{hyperref}
\hypersetup{colorlinks,linkcolor=blue,citecolor=blue,urlcolor = violet}
\usepackage[capitalise]{cleveref}

\numberwithin{equation}{section}
\usepackage{fancyhdr}
\title{Stopping Times
of Boundaries: \\ Relaxation and Continuity}

\author{H. Mete Soner\footnote{Department of Operations Research and Financial
Engineering, Princeton University, Princeton, NJ, 08540, USA, email: 
{\tt soner@princeton.edu}. Research partially supported by the National Science Foundation grant DMS 2406762.}
\and Valentin Tissot-Daguette\footnote{Department of Operations Research and Financial
Engineering, Princeton University, Princeton, NJ, 08540, USA, email: 
{\tt v.tissot-daguette@princeton.edu}}}

\date{\today}
\usepackage{oubraces}
\begin{document}
\maketitle

\input{Abstract}
\input{Introduction}

\input{Problem}
\input{Relaxation}
\input{Continuity}

\input{convSequential}
\input{Algorithm}

\input{AlgorithmLSC}
\input{Conclusion}
 \appendix
 \input{Appendix}

\renewcommand\bibname{References}
\bibliography{ref.bib}

\end{document}

%% file: Math_Operators.tex
\usepackage{xcolor}
\newcommand{\rr}[1]{\textcolor{black}{#1}}

\def\be{\begin{equation}}
\def\ee{\end{equation}}

\def\bea{\begin{align}}
\def\eea{\end{align}}

\def\bea*{\begin{align*}}
\def\eea*{\end{align*}}

\def\cref#1{Figure~\ref{#1}}
\theoremstyle{plain}
\newtheorem{theorem}{Theorem}[section]
\newtheorem*{theorem*}{Theorem}

\newtheorem{lemma}{Lemma}[section]
\newtheorem{corollary}{Corollary}[section]
\newtheorem{proposition}{Proposition}[section]

\newtheorem{asm}{Assumption}[section]
\theoremstyle{definition} 
\newtheorem{definition}{Definition}[section]
\newtheorem{example}{Example}
\newtheorem{remark}{Remark}[section]

\DeclareMathOperator{\E}{\mathbb{E}}
\DeclareMathOperator{\F}{\mathbb{F}}

\DeclareMathOperator{\N}{\mathbb{N}}

\DeclareMathOperator{\Q}{\mathbb{Q}}
\DeclareMathOperator{\R}{\mathbb{R}}

\DeclareMathOperator{\calC}{\mathcal{C}}

\DeclareMathOperator{\calE}{\mathcal{E}}
\DeclareMathOperator{\calF}{\mathcal{F}}

\DeclareMathOperator{\calK}{\mathcal{K}}
\DeclareMathOperator{\calL}{\mathcal{L}}

\DeclareMathOperator{\calO}{\mathcal{O}}
\DeclareMathOperator{\calP}{\mathcal{P}}

\DeclareMathOperator{\calS}{\mathcal{S}}
\DeclareMathOperator{\calT}{\mathcal{T}}

\DeclareMathOperator{\calX}{\mathcal{X}}

\DeclareMathOperator{\frakB}{\mathfrak{B}}

\DeclareMathOperator{\frakF}{\mathfrak{F}}

\DeclareMathOperator{\frakm}{\mathfrak{m}}
\DeclareMathOperator{\frakN}{\mathfrak{N}}

\DeclareMathOperator{\frakS}{\mathfrak{S}}

%% file: Abstract.tex
\vspace{0mm}
\begin{abstract}

We study the properties of the free boundaries and the corresponding hitting times 
in the context of  optimal stopping in discrete time. 
We first prove the continuity of the map from the boundaries to 
the expected value of the corresponding stopping policy both in
the supremum norm and also in a weaker, novel topology
induced by  the \textit{relaxed $L^{\infty}$ metric} that we introduce.
The latter is particularly useful when the optimal stopping boundary is 
only proved to be semicontinuous.   Secondly,
we study the connection between the hitting times,
and their relaxations as  widely employed in recent
numerical methods.  All these results 
together with the universal approximation capability 
of neural networks and the notion of   inf/sup  convolution
are then used
to provide a convergence analysis for the algorithm 
in  [Reppen, Soner, and Tissot-Daguette, Neural Optimal Stopping Boundary, 2025]  for the numerical
resolution of the exercise regions arising in 
the analysis of Bermudan type option. 

\end{abstract}
\vspace{1mm}
	
inf/sup convolution, convergence analysis,  Bermudan options. %
\\ \vspace{-3mm}


%% file: Introduction.tex
\section{Introduction}
The  optimal stopping problem, given {in its} simplest form  by
\begin{equation}\label{eq:OS0}
  \sup \left\{ v(\tau) := \E^{\Q}[ \varphi(\tau,X_{\tau})] \ : \ \tau = \text{ stopping time}\right\},
\end{equation}
for a {stochastic} process $X $ and  
a reward function $\varphi$,  
plays a central role in stochastic  optimal control and has applications  
{in numerous quantitative fields}. 
{In mathematical statistics, the theory of optimal stopping 
can be utilized to address sequential tests, such as the quickest detection of a Wiener process} 
\cite{ShiryaevStat, PeskirShiryaev}. 
In mathematical finance, optimal stopping problems emerge naturally from 
the pricing of American or Bermudan contingent claims 
\cite{elKarouiPardouxQuenez,Karatzas,PeskirShiryaev,Schweizer}. 
In this context, $\varphi$ is  the payoff of the claim, $X$ describes 
the price evolution of the underlying asset, and the optimal value in 
\eqref{eq:OS0} gives the initial price of the contract. 
{While the initial} motivation of this paper is to provide 
a convergence study of the neural  stopping boundary
algorithm of  \cite{ReppenSonerTissotDF,ReppenSonerTissotFB},
{our analysis is essentially applicable to all scenarios 
where the exit times play a central role. Consequently,
we present our findings in this broader context, then
analyze the mentioned algorithm using these general results.}

The unifying structure is the approximate parameterization of 
the \textit{stopping boundaries} 
by feedforward artificial neural networks,
and  the central question revolves around their asymptotic effectiveness.
In fact, the algorithm developed in \cite{ReppenSonerTissotFB}
computes the \textit{stopping} (or \textit{exercise}) \textit{region} 
for Bermudan options. 
\rr{Following \cite{ReppenSonerTissotFB},
we assume that in an appropriate coordinate system,
the optimal stopping region is the \textit{epigraph} 
of a function $f^{\diamond}$. In other words, 
stopping occurs the first time the underlying process lies above the boundary $f^{\diamond}$. } 
Consequently, the hitting time of  $f^{\diamond}$ is  optimal for \eqref{eq:OS0}. 
Hence, we can restrict the set of stopping times in \eqref{eq:OS0} to hitting times of boundaries.
Namely,
\begin{equation}\label{eq:OS0f}
  \sup \left\{ v(\tau_f) := \E^{\Q}[ \varphi(\tau_f,X_{\tau_f})] \ : \ f = \text{ stopping boundary} \right\},
\end{equation}
where $\tau_f$ is the hitting time of the epigraph (or hypograph) of $f$.
{In the neural boundary algorithm, networks approximate $f^{\diamond}$,
and the convergence analysis of this method
relies mainly on the properties of the relaxation of the hitting times,
and the continuity of the map from the boundaries 
to the expected value of the corresponding stopping policy,
as proved in \cref{sec:relaxation,sec:continuity}, respectively.}

It is well-known that, training with 
stochastic gradient descent of loss functions 
computed via hitting times results in vanishing gradient,
and as in the \textit{deep optimal stopping} of  \cite{Becker1,Becker2} and in \cite{ReppenSonerTissotFB},
relaxed stopping must be introduced.
Relegating the technical details to \cref{sec:relaxation}, 
the relaxed version of problem \eqref{eq:OS0}  is given by,
\begin{equation}\label{eq:relax0}
    \sup\left\{ v(\mu) := \E^{\Q}\left[\int \varphi(t,X_{t}) \mu(dt) \right] \ : \ \mu =  \text{ random probability measure}\right\}.
\end{equation}

Also called \textit{convexification} \cite{Acikmese} or \textit{compactification}     
\cite{ElKaroui}, relaxation procedures such as \eqref{eq:relax0} are 
classically used in stochastic control problems to prove the existence of an optimal solution \cite{FlemingDomokos,Haussmann_Lepeltier,McShane}. 
Indeed, {relaxation convexifies the problem allowing 
for numerous methods and tools for its analysis. 
Moreover, as we prove in  \cref{prop:relaxOS},
both  \eqref{eq:OS0} and \eqref{eq:relax0} achieve the same optimal value. 
Indeed, this holds essentially in all other} contexts; 
see, e.g.,  \cite{Gyongy,Hobson} and \cite[Lemma 1.5.2.]{Krylov} 
for continuous time  optimal stopping, and 
\cite{ElKaroui,FlemingDomokos,Haussmann_Lepeltier} for general stochastic optimal control.  

\rr{In the neural optimal stopping boundary method,
the stopping policy is relaxed by introducing a  so-called \textit{fuzzy} region around the boundary
with a small width  $\varepsilon>0$. Concretely, the fuzzy region separates the continuation and stopping regions through a gradual transition: 
the stopping rule remains the same  outside the fuzzy region, while inside, 
the usual binary decision ("exercise or continue")  is replaced by stopping intensities  proportional to the signed distance to the boundary. }
This procedure generates   \textit{relaxed} hitting times of boundaries$-$namely random  
probability measures supported  on the set of exercise dates. 
The relaxed problem \eqref{eq:relax0} thus becomes
\begin{equation}\label{eq:relax1}
\sup\left\{ v(\mu_f^\varepsilon) := \E^{\Q}\left[\int \varphi(t,X_{t}) 
{\mu_f^\varepsilon}(dt) \right] \ : \   {\mu_f^\varepsilon} =  \text{relaxed hitting time of $f$}\right\}.
\end{equation}
As shifting the boundary ever so slightly changes the stopping intensities inside the fuzzy region, the above relaxation makes stochastic gradient methods applicable. \rr{The fuzzy region is akin to the \textit{mushy zone}  
encountered in phase-field models \cite{B-mushy,Soner,V-mushy}, a physical phenomenon characterized by the  coexistence of two phases (typically solid and liquid) within a material. } 
The introduction of a  fuzzy  region is critical in other deep learning algorithms as well,  such as  occupation networks in  computer vision  \cite{OccupancyNN} or deep level-set functions for solving Stefan problems  \cite{SST}. 

Our central results of continuity are proved in 
\cref{sec:continuity}. In particular, \cref{thm:unifcontV,thm:lscV} establish the continuity
of the value functional $\frakF \ni f \mapsto v(\tau_f)$, 
where $\frakF$ is a set of stopping boundaries
{\rr{that are lower semi-continuous}}. 
Formally, these results are
in the following form.
{\rr{\begin{theorem*}\label{thm:unifcontV0}
\textnormal{\textbf{(Continuity of $\boldsymbol{f \mapsto v(\tau_f)}$)}} 
Let $\frakm$ be a metric on $\frakF$. 
Then for all $\delta>0$, there exists $\iota  >0$ such that for all $f, g \in \frakF$,
\begin{equation*}
\frakm(f,g) \le \iota \ \Longrightarrow   \ |v(\tau_{f}) - v(\tau_{g})| \le \delta. 
\end{equation*}
\end{theorem*}
\noindent
Indeed, while \cref{thm:unifcontV} employs the supremum norm,
we consider a weaker topology induced by the \emph{relaxed $L^{\infty}$ metric} 
of Definition \ref{def:relaxedSupDist} 
in \cref{thm:lscV} to accommodate  boundaries that are only semi-continuous. 
Concretely, this metric is defined as 
$$
\frakm(f,g):= \frakm_d(f,g)+\frakm_d(g,f), \quad  \; \; 
\frakm_d(f,g)
:=  \inf_{\psi\in \Psi}  
\sup_{\xi \in \calE} \ \big[(f(\psi(\xi))- g(\xi))^+ +  |\psi(\xi)-\xi| \big],
 \label{eq:relaxDistH}
$$
where $\Psi$ contains all endofunctions $\psi:\calE \to \calE$.
The above metric discussed further in subsection \ref{sec:relaxLinfty} is 
similar to the Skorokhod $J1$ distance  for càdlàg processes,
and compares  boundaries in $\frakF$ uniformly, 
possibly after a local transformation of the domain. 
It is also inspired by the inf/sup convolution,  which is an effective regularization 
tool widely used in the theory of viscosity solutions  \cite{FS, LasryLions}. We shall prove in \cref{lem:metric} that the "directed" map 
$\frakm_d$ is an \emph{asymmetric metric}; see  \citet{mennucci}.}} 

An immediate corollary, \cref{cor:relaxUnif}, 
{\rr{to the continuity results}}
is the convergence  of the value \eqref{eq:relax1}
associated to the relaxed stopping times
to the value using the same interface  as the width $\varepsilon$ of the fuzzy region tends to zero.   

 In \cref{sec:NOSB}, we combine the above results and provide a convergence 
 analysis of the \textit{neural optimal stopping boundary} algorithm \cite{ReppenSonerTissotFB}. 
Briefly, this algorithm looks for a \textit{neural stopping boundary}, 
namely a \textit{feedforward neural network} $g^{\theta}$   
such that 
$$
{v(\mu_{g^\theta}^\varepsilon)} \approx v(\tau_{g^{\theta}}) \approx \sup_{f\in \frakF}
v(\tau_f) \  = v(\tau_{f^{\diamond}}).
$$ 
The parameter vector $\theta$ is then trained by 
maximizing the relaxed reward $v(\mu_{g^\theta}^\varepsilon)$ using stochastic gradient ascent.  
The following result, stated  precisely in \cref{thm:convCont}, implies
the  convergence of the algorithm employed in \cite{ReppenSonerTissotFB}.
\begin{theorem*}\label{thm:convCont0} \textnormal{\textbf{(Convergence)}}
For all  $\delta>0$,  there exists $\varepsilon>0$ and a feedforward  
neural network  $g^{\theta}$ such that ${v(\mu_{g^{\theta}}^\varepsilon)}\ge v(\tau_{f^{\diamond}}) - \delta$. 
\end{theorem*}
\noindent
When the optimal stopping boundary $f^{\diamond}$ is continuous, 
one can apply the  Universal Approximation 
Theorem \cite{Cybenko,Hornik} stating  that any continuous function can be 
approximated arbitrarily well by  feedforward neural networks in  compact sets. 
This together with the continuity of $f \mapsto v(\tau_f)$,
implies the above result.
However, the optimal stopping boundary $f^{\diamond}$ 
may fail to be continuous, preventing the use of the Universal Approximation Theorem. 
{Nevertheless,  under mild assumptions, the stopping boundary is  
shown to be semicontinuous in \cref{prop:lscBdry}.  

Motivated by this, we first extend in \cref{thm:LSCUAT} the Universal Approximation Theorem 
by proving that semicontinuous functions can be 
approximated by neural networks  in the 
\textit{relaxed $L^{\infty}$ metric} introduced in  Definition \ref{def:relaxedSupDist}. 
This is achieved through inf-sup convolutions
recalled in \cref{sec:convLSC}.   
This approximation together
with  \cref{thm:lscV} and the semicontinuity of $f^\diamond$
enables us to prove the above 
theorem. 
Incidentally, we suggest that the universal approximation
proved in \cref{thm:LSCUAT},}
is an interesting result in its own right 
and can be applied to other tasks  involving the approximation of  semicontinuous functions.

There is a recent body of literature on the convergence of deep learning methods in stochastic optimal control:  \citet{Hure} study the convergence  of discrete time stochastic control problems involving policy and/or value iteration; \citet{LapeyreLelong} prove the convergence of the Longstaff-Schwartz algorithm \cite{LSMC} using neural networks to approximate the continuation value; 
\citet{Gonon} sheds light on the expressivity and maximal size  of   neural networks to produce optimal stopping rules.  
{Additionally,}  \cite[Proposition 4]{Becker1} of the \textit{deep optimal stopping} paper 
 provies the near optimality of deep stopping strategies. 
 {A common property used in all these papers is dynamic programming,
 while} the neural optimal stopping boundary algorithm 
  \cite{ReppenSonerTissotFB} constructs an optimal stopping rule in a forward fashion.
  Consequently, our convergence analysis  necessitates 
  other techniques and tools, and differs greatly from previous works.  

\textbf{Structure of the paper.} In \cref{sec:OS}, we introduce the optimal stopping problem,  our main structural assumption on the exercise region, and give financial examples. In \cref{sec:relaxation}, we define the fuzzy region and construct relaxed stopping rules. \cref{sec:continuity} establishes the  continuity of the map $f\mapsto v(\tau_f)$ as well as important consequences. In \cref{sec:NOSB}, we prove the convergence of the neural optimal stopping boundary method in the continuous and semicontinuous case, respectively.  \cref{sec:conclusion} provides  concluding remarks and \cref{app:relaxedproblem} contains the proof of \cref{prop:relaxOS}. 

%% file: Problem.tex
\section{Optimal stopping}\label{sec:OS}
 Let $\calT \subset [0,\infty)$  a  finite subset of  time indices containing $0$ and
$(\Omega, \F = (\calF_t)_{t \in \calT},\Q)$ a filtered probability space. 
The state $X = (X_t)_{t\in \calT}$
is an $\F$-adapted, Markov process which take values in a Euclidean space $\calX$. We assume for concreteness that $\F$ is the  filtration generated by $X$.
 Given a \emph{reward (or payoff) function} $\varphi:\calT \times \calX \to \R$,  the associated \emph{optimal stopping problem}  is to 
\begin{equation}
\label{eq:OS}
\text{maximize}
\ \ 
v(\tau):=  \E^{\Q}[ \varphi(\tau,X_{\tau})],
\qquad
\text{over}\  \tau \, \in \, \vartheta(\calT),
\end{equation} 
where 
$\vartheta(\calT)$ is the set of 
$\calT$-valued $\F$-stopping times.
Problem $\eqref{eq:OS}$ is well-defined (and non-trivial) if   $\max_{t\in \calT}|\varphi(t,X_t)| \in L^1(\Q)$. We write $v^{\diamond}$ for the optimal  value of \eqref{eq:OS}. 
The following assumption on $\varphi$ and the dynamics of $X$ will prove useful.  
\begin{asm}
\label{asm:growth}
There exists $a>0$ 
such that 
$\varphi(t,\cdot) \in \calL_a \ \forall \ t \in \calT$, where 
\be
\label{e.lca}
\calL_{a}:= \left\{ \phi \in \calC(\calX;\R) \ : \ 
|\phi(x) | \le C[1+|x|^a]\ \
\text{for some}\ C>0 \right\}.
\ee
Also, we suppose that $\phi \in \calL_a \Longrightarrow K_{t,s} \phi \in \calL_a$ $\forall \  t, s \in \calT$, with the transition semigroup
\begin{equation}\label{eq:semigroup}
    K_{t,s} \phi(x)=
\E[ \phi(X_{s}) \mid X_{t}=x], \quad x\in \calX, \quad t<s. 
\end{equation}
\end{asm}

Allowing the second component of  $K_{\cdot,\cdot}$  in \eqref{eq:semigroup} to be a stopping time,  
define  the value function  
\begin{equation}
\label{eq:vdiamond}
    v^{\diamond}(t,x)= \sup_{t \le  \tau \in \vartheta(\calT )}
K_{t,\tau} \varphi(\cdot, x) \ , 
\end{equation}
and stopping region
$\calS_t = \{x\in \calX  :   \varphi(t,x) = v^{\diamond}(t,x)\}, \ t\in \calT. $
We then have the following result, proved in \cite{ReppenSonerTissotFB}. 
\begin{proposition}
\label{prop:closedStopRegion} Under Assumption \ref{asm:growth},
the value function $v(t,\cdot)$ belongs to $ \calL_a$ for each $ t\in \calT$. In particular, $v^{\diamond}(t,\cdot)$ is continuous and the stopping region $\calS_t$ is a relatively closed subset of $\calX$
.
\end{proposition}

Similar to   \cite{ReppenSonerTissotFB}, we make the following structural  assumption  
yielding the existence of a  \textit{stopping boundary} (or \textit{free boundary})   delimiting the stopping region in a problem specific coordinate system. 
\begin{asm}
\label{asm:star}
There exist  measurable functions
$$
\alpha: \calX \to \R_+,
\quad
\Xi: \calX \to \calE := \Xi(\calX),
\quad
 f^{\diamond}: \calT \times \ \calE \to [0,\infty],
 $$
 such that $A = (\Xi,\alpha): \calX \to  \calE \times \R_+$ is a homeomorphism
 and for every $t \in \calT$,
\be
\label{e.st}
\calS_t = \{  x \in  \calX \ :\
\eta \left(f^{\diamond}(t,\Xi(x))-\alpha(x)\right) \le 0  \}, \qquad \eta \in \{-1,1\}. 
\ee 
\end{asm}
We suppose that the \textit{latent space} $\calE$ is itself  a subset of a Euclidean space (e.g. $\R^m$ for some $m\ge 1$) and write $|\xi|$ for the norm of $\xi\in \calE$. 
 Under the above assumption, the \textit{optimal stopping boundary} is thus given by 
\begin{align}\label{eq:stopBdry}
    f^{\diamond}(t,\xi) = \begin{cases} 
    \inf\{a \ge 0 \ : \ A^{-1}(\xi,a) \in \calS_t\}, & \eta = 1, \\[0.5em]
    \sup\{a \ge 0 \ : \ A^{-1}(\xi,a) \in \calS_t\}, & \eta = -1.
    \end{cases}
\end{align}
\begin{figure}[t]
\caption{Stopping region of a max-call option on two symmetric assets. The upper connected component becomes an epigraph through $(\alpha,\Xi)$ \rr{(right panel), where we   identify $\Xi(x) = (\frac{x_1}{x_2},1) \in \R^2$ with its first component.} Figures adapted from \cite{ReppenSonerTissotFB}.  } 
\vspace{-2mm}
\begin{subfigure}[b]{0.49\textwidth}
    \centering
    \caption{Stopping region in $\calX$}
    \includegraphics[height=2.0in,width=2.3in]{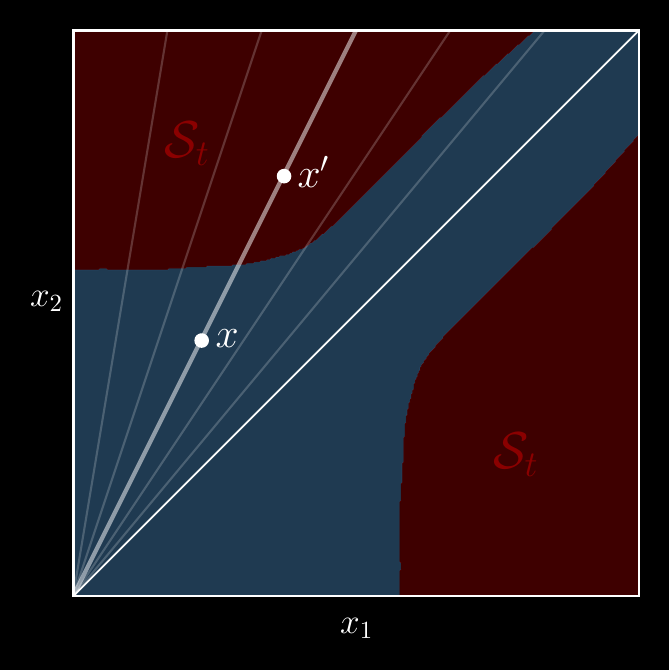}
    \label{fig:X}
\end{subfigure}
\begin{subfigure}[b]{0.49\textwidth}
    \centering
    \caption{Stopping region through $(\alpha,\Xi)$}
    \includegraphics[height=2.0in,width=2.3in]{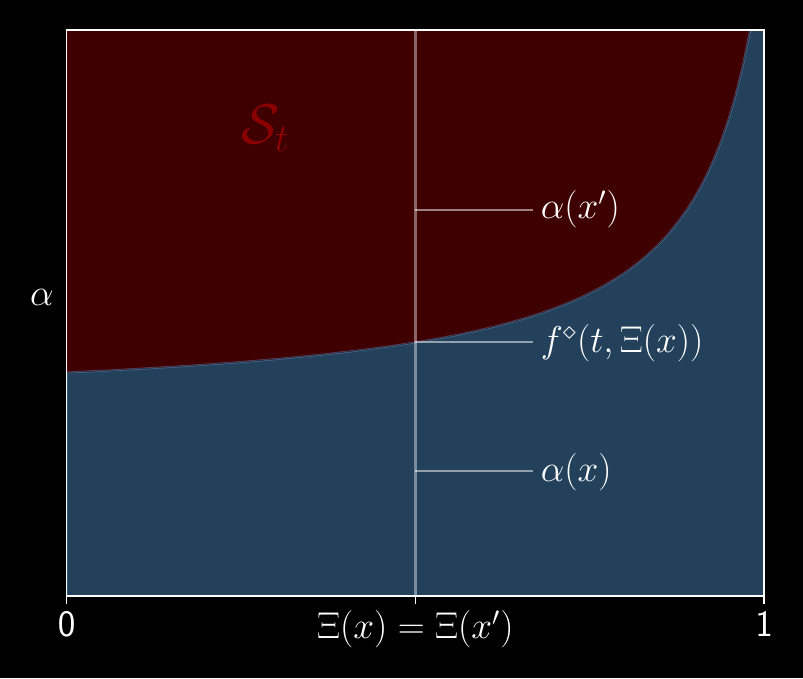}
    \label{fig:homeo}
\end{subfigure}
\label{fig:homeomorphism}
\end{figure}  
Consequently, the time $t$ section of the stopping region is the epigraph  of $f^{\diamond}(t,\cdot)$ when  $\eta =1$; see \cref{fig:homeomorphism}. Similarly, $\calS_t$ is the hypograph of $f^{\diamond}(t,\cdot)$ when $\eta =-1$. We refer the reader to \cite{BroadieDetemple} for further  information on the structure  of the stopping region and to \cite{Laurence} for a deep proof of the regularity of the boundary.  Additionally, classical books \cite{Velichkov,Petrosyan}  provide  a thorough treatment of the free boundary regularity in one-phase problems and respectively, in obstacle problems. 
For completeness, next we prove the following  simple properties  of $f^{\diamond}$. 
\begin{proposition} \label{prop:lscBdry}
    For every $t\in \calT$, the map $\xi \mapsto f^{\diamond}(t,\xi)$ is lower (respectively upper) semicontinuous when  $\eta =1$ (resp. $\eta = -1$). Moreover, if the payoff function $\varphi$ is convex in $x$, then $f^{\diamond}$ is locally Lipschitz continuous. 
\end{proposition}
\begin{proof}
   Fix $t\in \calT$ and prove the result for $\eta=1$. Similar arguments can be applied for the case   $\eta=-1$.  
    Let $\xi \in \calE$ and $(\xi_n) \subset \calE$ such that $\xi_n \to \xi$ and fix $\varepsilon >0$. In light of $\eqref{eq:stopBdry}$,    there exists $a_n \in \R_{+}$ such that $x_n := A^{-1}(\xi_n,a_n)\in \calS_t$ and $f^{\diamond}(t,\xi_n) \ge a_n - \varepsilon$ for every $n\in \N$.  Defining $a = \limsup_{n\to \infty} a_n,$ then 
    \begin{equation*}
        \liminf_{n\to \infty} f^{\diamond}(t,\xi_n) \ge a - \varepsilon. 
    \end{equation*}
    Next, the continuity of $A^{-1}$ yields  $\lim_{k\to \infty} x_{n_k} = x := A^{-1}(\xi,a) $ for a subsequence $(n_k)$ such that $a_{n_k} \to a$.  As $(x_{n_k}) \subset \calS_t$ and  $\calS_t$ is closed (see \cref{prop:closedStopRegion}), then  $x \in \calS_t$. Therefore $a \ge f^{\diamond}(t,\xi)$, giving  
    $\liminf_{n\to \infty} f^{\diamond}(t,\xi_n) \ge f^{\diamond}(t,\xi) - \varepsilon.$ 
   Letting $\varepsilon\to 0$ indeed shows the lower semicontinuity of  $\xi \mapsto f^{\diamond}(t,\xi)$. 
    
For the second assertion, it is shown in \cite{BroadieDetemple} that the convexity of  $x \mapsto \varphi(t,x)$ carries over to $\calS_t$.  In particular,  $f^{\diamond}$ is   locally Lipschitz continuous.
\end{proof}

In the remainder of the paper, 
to simplify the exposition we assume  that $\eta =1$, and the case $\eta = -1$ follows \emph{mutadis mutandis}. For fixed maps $(\Xi,\alpha)$ that satisfies Assumption \ref{asm:star} with $\eta =1$,  the corresponding optimal stopping boundary $f^{\diamond}(t,\xi)$ is therefore  lower semicontinuous (l.s.c.)~in $\xi$ for all $t\in \calT$, and we set 
\begin{equation}
\label{eq:frakF}
\frakF := \{f:\calT \times \ \calE \to [0,\infty], \;  f(t,\cdot) \text{ l.s.c. } \forall t \ \in \calT\}.
\end{equation}
We associate to each $f\in \frakF$ the \emph{stopping region}
\begin{equation}
    \frakS_t(f) = \{x\in \calX \ : \ \alpha(x) \ge f(t,\Xi(x)) \},
\end{equation}
and  the \emph{hitting time}
$\tau_f = \inf\{t \in \calT  :  X_t \in \frakS_t(f)\}$.  Assumption \ref{asm:star} implies in particular that it is enough to consider hitting times of  boundaries to achieve the optimal value in  \eqref{eq:OS}. Namely,
$$
v^{\diamond} = \sup_{f\in \frakF}v(\tau_f)  = v(\tau_{f^{\diamond}}),
$$  
as already shown in \cite[Lemma 2.3]{ReppenSonerTissotFB}. Problem \eqref{eq:OS} is therefore  equivalent to finding $f\in \frakF$ such that $\frakS_t(f) = \calS_t$ for all $t \in \calT$. 

\subsection{Financial examples}\label{sec:finEx}
Suppose that $X = (X^1,\ldots,X^m)$ describes the price process of $m\ge 2$ financial assets. The state space is  therefore  $\calX = \R^m_+$.   Moreover,  consider payoff functions of the form 
\begin{equation}
\varphi(t,x) = e^{-rt}(\alpha(x)-\kappa)^+,
\end{equation}
where $r\ge 0$ is a constant interest rate, $\kappa>0$ is the \textit{strike} of the option,  and $\alpha:\calX\to [0,\infty]$ is a statistic of the underlying prices.  
For this class of contracts,    Assumption \ref{asm:star} holds  with  $\eta=1$, and the homeomorphism $A = (\Xi,\alpha)$,  $\Xi(x)=\frac{x}{\alpha(x)}$. See Theorem 4.1 in \cite{ReppenSonerTissotFB} for a  complete  statement. 

\begin{example}\label{ex:maxcall}
     Consider a \textit{max-call option}, which is obtained by choosing  $\alpha(x)  = \max_{i=1,...,m}x^i$. Hence $\varphi$ is convex in $x$, and \cref{prop:lscBdry} tells us that the stopping boundary is locally Lipschitz continuous.   Moreover, we have the latent space   $\calE =  \{\xi \in [0,1]^m  :  \max_{i=1,...,m}\xi_i=1  \}$; see  Assumption \ref{asm:star}. In particular, it is compact. Interestingly, the stopping region consists of $m$ connected components, as can be seen in \cref{fig:X} in the case of $m=2$ assets. The right panel (\cref{fig:homeo}) displays  the image through $(\alpha,\Xi)$ of the upper connected component of the stopping region in \cref{fig:X}. Although $\calE$ is a subset of $[0,1]^2$, we can    identify $\Xi(x) = (\frac{x_1}{x_2},1)$ with its first component when $x_2 \ge x_1$, as depicted in \cref{fig:homeo}.  

     Max-call options are popular benchmark examples for numerical methods addressing high-dimensional optimal stopping problems; see \cite[Section 8]{ReppenSonerTissotFB} and the references therein. Indeed,  the maximum function in the payoff makes the problem  truly high-dimensional in the sense that the process $\alpha(X)$ cannot be replaced by a one-dimensional diffusion with identical law. 
\end{example}

\begin{example} \label{ex:minmax} 

If $\alpha(x) = \min_{i=1,...,m}x^i$, then the associated contract is  a so-called \textit{min-call option}.  
Note that the payoff $\varphi(t,x)$ is \textit{not} convex in $x$, leading to pathological properties of the stopping region prior to maturity as  
 studied by \citet{DetempleMin} in the case $m=2$. In particular,   $\partial \calS_t$ contains a nonempty singular set given by a segment on the diagonal $\{x_1 = x_2\}$; see \cref{fig:minX}.
Consequently, the boundary $\xi \mapsto f^{\diamond}(t,\xi)$ is discontinuous at $\xi = (1,1)$, hence strictly (lower) semicontinuous; see \cref{fig:minae}. 
\end{example}

\begin{figure}[H]
\caption{Stopping region of a min-call option on two symmetric assets (one time section). We note that the stopping boundary is discontinuous at $\Xi(x) = (1,1)$. }
\vspace{-2mm}
\begin{subfigure}[b]{0.49\textwidth}
    \centering
    \caption{Stopping region in $\calX$}
    \includegraphics[height=2.0in,width=2.3in]{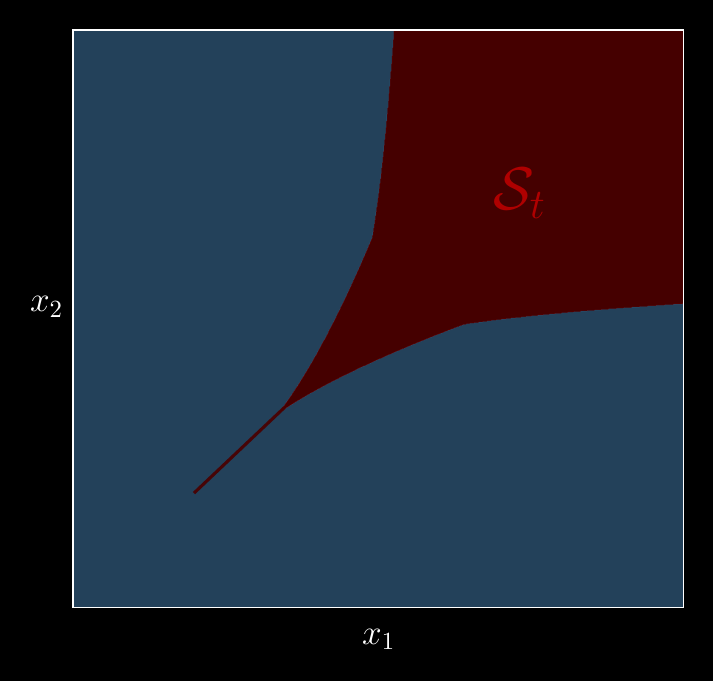}
    \label{fig:minX}
\end{subfigure}
\begin{subfigure}[b]{0.49\textwidth}
    \centering
    \caption{Stopping region through $(\alpha,\Xi)$}
    \includegraphics[height=2.0in,width=2.3in]{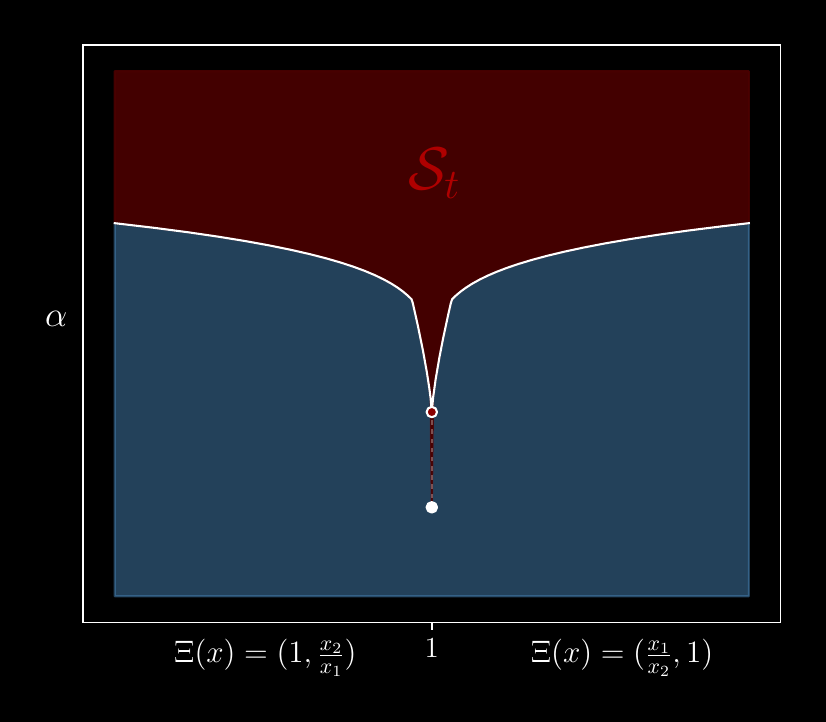}
    \label{fig:minae}
\end{subfigure}
\label{fig:mincall}
\end{figure}  

%% file: Relaxation.tex
\section{Relaxed stopping rules}\label{sec:relaxation}


To find an optimal boundary with stochastic gradient techniques as  in \cref{sec:NOSB}, the optimal stopping problem needs to be relaxed. Indeed, the map $f \mapsto \tau_f(\omega)$ is piecewise constant for $\Q$-almost every $\omega \in \Omega$ as the hitting time of $f$, taking values in the discrete set $\calT$,   is typically unchanged when  the stopping  boundary varies slightly. This in turn leads to severe vanishing gradient problems. We overcome this difficulty by approximating $\tau_f(\omega)$ by a sequence of relaxed  stopping rules based on \textit{fuzzy regions} parametrized by $\varepsilon>0$, 
\begin{equation}\label{eq:fuzzy1S}
\frakB_t^{\varepsilon}(f) := \frakS_t(f-\varepsilon) \setminus \frakS_t(f) = \{x\in \calX \ : \ d(t,x;f) \in (0, \varepsilon] \}, 
\end{equation}
where $d:\calT\times \calX \times \frakF \to [0,\infty]$ is the distance from $x$ to $\frakS_t(f)$ in the coordinate system $(\Xi,\alpha)$,
\begin{equation}\label{eq:distFct}
    d(t,x;f) =   (f(t,\Xi(x)) - \alpha(x))^{+}, \quad (t,x,f)\in \calT \times \calX \times \frakF,  
\end{equation} 
and an illustration is given in the left panel of \cref{fig:fuzzy&phaseIndicator}. 

We continue by first defining the general relaxed stopping rules, and then  the specific one induced by the fuzzy region. Let $\calP(\calT)$ be the set of probability measures on $\calT$ and $\underline{\vartheta}(\calT)$ denotes the set of $\calP(\calT)-$valued  random variables $$\mu(\omega) = \sum_{t\in \calT} P_t(\omega) \delta_t, \quad \omega \in \Omega, \qquad (\delta_t: \text{Dirac measure at $t$})$$ 
such that the process $(P_t)_{t\in \calT}$ is  $\F$-adapted. 
In this context, $\underline{\vartheta}(\calT)$ is the set of \textit{relaxed} (or \textit{randomized}) $\F$-stopping times taking values in $\calT$. If we identify $\tau \in \vartheta(\calT)$ with 
$\delta_{\tau} \in \underline{\vartheta}(\calT)$, then   $\underline{\vartheta}(\calT)$ contains all  $\F$-stopping times. Indeed, this procedure is effectively  a \textit{convexification} of the stopping times as every relaxed stopping rule is  a random convex combination of the Dirac measures $(\delta_t)_{t\in \calT}$.  
We can now define the value functional
\begin{equation}\label{eq:relaxValue}
v(\mu) = \E^{\Q}\left[\int_{\calT} \varphi(t,X_{t}) \mu(dt) \right],\quad  \mu \in \underline{\vartheta}(\calT).
\end{equation}
The \emph{relaxed optimal stopping problem} amounts to maximizing $v(\mu)$ over all $\mu \in \underline{\vartheta}(\calT)$. 
Importantly,  the  relaxed  problem and classical one   \eqref{eq:OS} achieve the same optimal value as we prove in the next proposition. 
Similar results in continuous time can be found in \cite{Gyongy}, \cite{Hobson} and \cite[Lemma 1.5.2]{Krylov}. 

\begin{proposition}\label{prop:relaxOS}
Let $v^{\diamond}$ be the  value of the classical optimal stopping problem \eqref{eq:OS}. 
Then, 
\begin{equation}\label{eq:OSvsRelax}
      \sup_{\mu \in \underline{\vartheta}(\calT)} v(\mu)  = v^{\diamond}.
\end{equation}
 Moreover, any maximizer $\mu^{\diamond}$
 of $ \mu \mapsto v(\mu)$ is a  stopping time, i.e. $\mu^{\diamond} = \delta_{\tau^{\diamond}} $ for some $\tau^{\diamond} \in \vartheta(\calT)$. 
\end{proposition}
\begin{proof}
See \cref{app:relaxedproblem}.
\end{proof}
  
We next construct a relaxed stopping rule from a given boundary $f\in \frakF$ and a fuzziness parameter $\varepsilon>0$. Set $\calT_{t-} = \calT \cap \  [0,t)\ $ for $t>0$ and $\calT_{0-} = \varnothing$, and recursively define the $[0,1]-$valued  stochastic process $P^{\varepsilon}_f = (P_{f,t}^{\varepsilon})_{t\in \calT}$  representing  stopping probabilities by,
\begin{align}\label{eq:stopProb}
     P^{\varepsilon}_{f,t} = \underbrace{p_{f,t}^{\varepsilon}}_{\text{stop at $t$}} \ \ \underbrace{\Big(1- \sum_{s \in \calT_{t-}} P_{f,s}^{\varepsilon}  \Big)}_{\text{continue until $t$}}, \qquad 
     p_{f,t}^{\varepsilon} = (\chi^{\varepsilon}\circ d)(t,X_t;f), 
\end{align}
with the relaxed phase indicator function $\chi^{\varepsilon}(\delta) = (1 - \delta/\varepsilon)^+ \wedge 1$. 
Specifically, $\chi^{\varepsilon}$ is equal to one in the stopping region $\frakS_t(f)$, it is linearly decreasing in the fuzzy region $\frakB_t^{\varepsilon}(f)$, and is zero otherwise. The right panel of \cref{fig:fuzzy&phaseIndicator} displays the composition $\chi^{\varepsilon} \circ d$ in the space $\Xi(\calX) \times \alpha(\calX)$. 
We can now define the relaxed stopping rule associated to $f$ as 
\begin{equation}\label{eq:relaxedRule}
    \mu^{\varepsilon}_f(\omega) = \sum_{t\in \calT} P^{\varepsilon}_{f,t}(\omega) \  \delta_{t}. 
\end{equation}

\begin{figure}[t]
\caption{Left panel: fuzzy region (purple) separating the interior of the continuation region (blue) from the stopping region (red). Right panel: Illustration of $\chi^{\varepsilon} \circ d(\cdot; f)$ in the coordinates $(\Xi,\alpha)$, where $d(\cdot; f)$ is the distance function to $\frakS_t(f)$ and  $\chi^{\varepsilon}(\delta) = (1 - \delta/\varepsilon)^+ \wedge 1$.} 
\vspace{-1mm}
\begin{subfigure}[b]{0.49\textwidth}
    \centering
    \includegraphics[height=2.0in,width=2.5in]{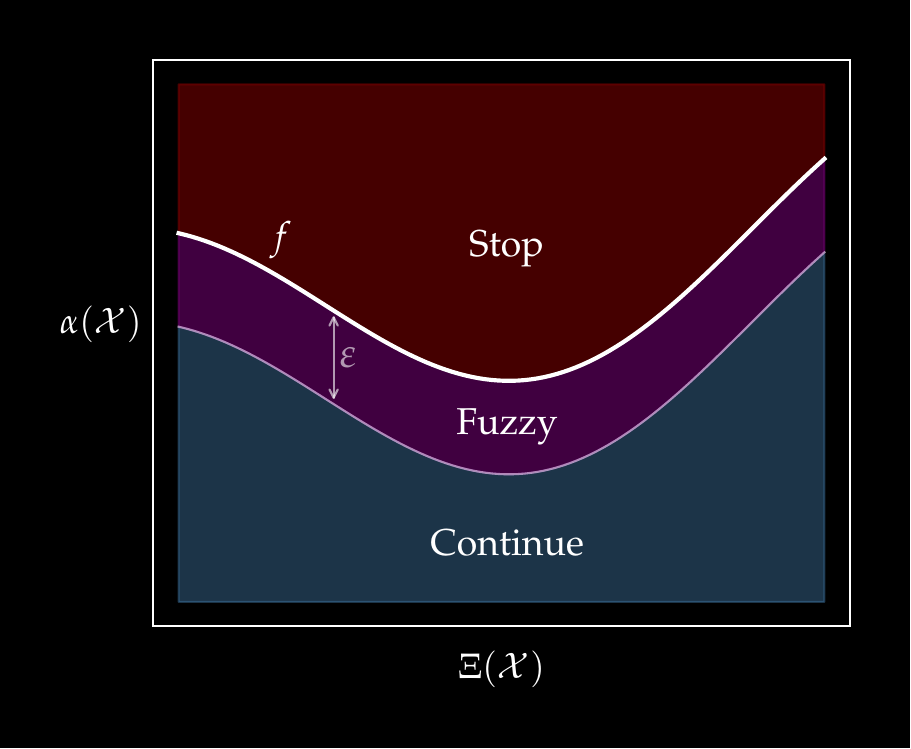}
    \label{fig:fuzzybrdy}
\end{subfigure}
\begin{subfigure}[b]{0.49\textwidth}
    \centering
         \includegraphics[height = 2.0in,width = 2.5in]{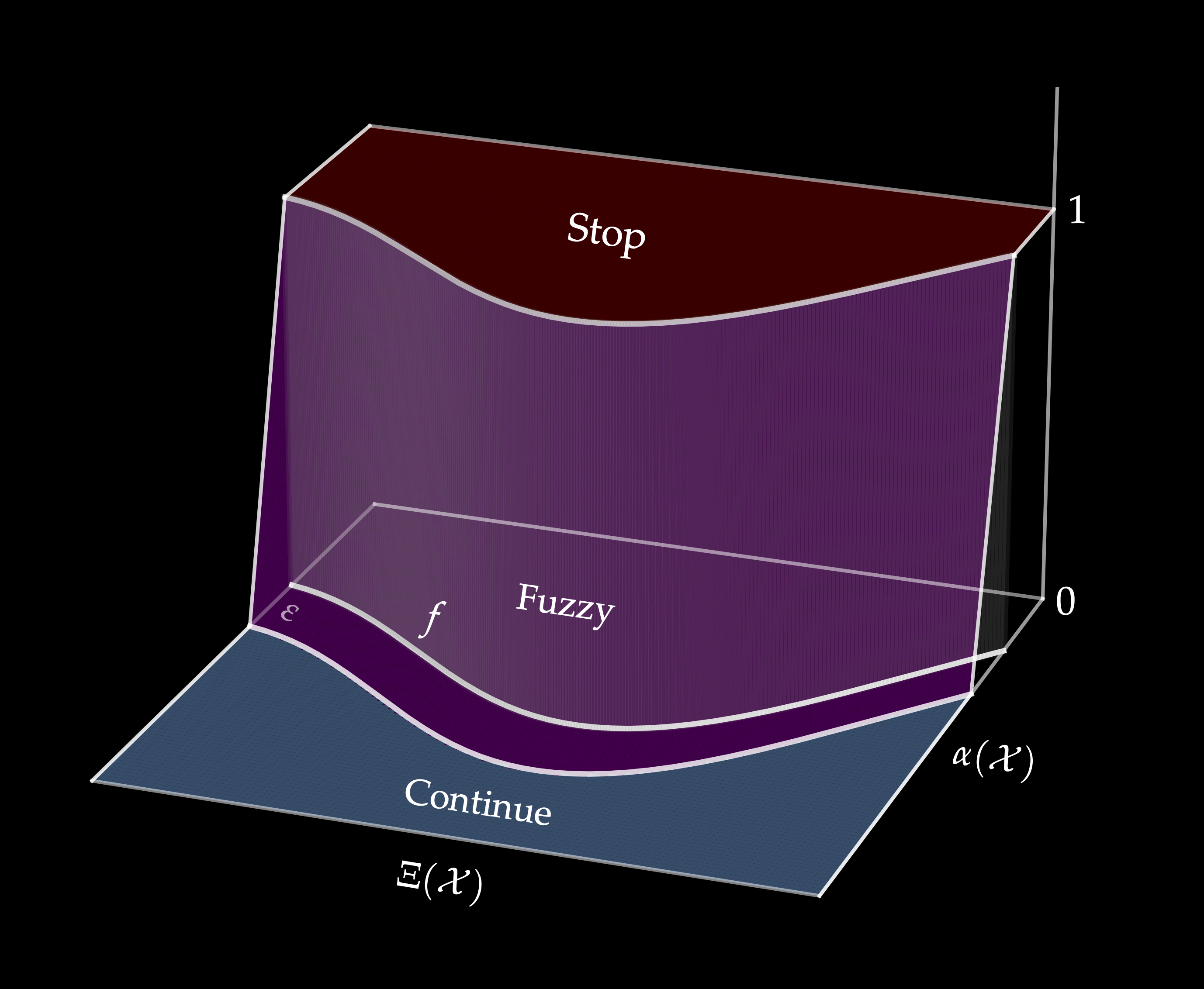}
     \label{fig:hFunction}
\end{subfigure}
\label{fig:fuzzy&phaseIndicator}
\end{figure}  
\noindent
Using \eqref{eq:stopProb}, we directly verify that $\mu_f^{\varepsilon}(\omega)(\calT) = 1$. Additionally, as for each $t$, $P^{\varepsilon}_{f,t}$ depends only on $\{X_s \ : \ s \in \calT\cap \  [0,t] \}$, $P^{\varepsilon}_{f}$ is adapted to $\F$. Hence,  $\mu_f^{\varepsilon}$ is a relaxed $\F-$stopping time.

To complete the approximation, in \cref{cor:relaxUnif} we prove the following convergence which was already argued in \cite[Lemma B.1]{ReppenSonerTissotFB},
\begin{equation}\label{eq:relaxConv}
    v(\mu_{f}^{\varepsilon}) \longrightarrow v(\tau_{f}) \quad \text{as } \ \varepsilon \downarrow 0.
\end{equation}
That is, the value of the relaxed stopping rule  $\mu_f^{\varepsilon}$ converges to the expected reward  from  the hitting time of $f$ when the fuzzy parameter $\varepsilon$ decreases to zero. 

The above result is  a byproduct of the continuity of  $f\mapsto v(\tau_f)$, proved in the next section.   
Towards our main results \cref{thm:unifcontV,thm:lscV}, we start with a  
continuity lemma for relaxed stopping rules with respect to  
the \textit{total variation distance} of probability measures, 
$$  \textnormal{TV}(\mu,\mu') = \frac{1}{2}\sum_{t\in \calT} |\mu(\{t\}) - \mu'(\{t\})|\in [0,1], \qquad \mu, \mu' \in \calP(\calT).$$
Note that when  $\mu,\mu' \in \underline{\vartheta}(\calT)$, then  $\textnormal{TV}(\mu,\mu')$ is the random variable  $ \omega \mapsto \textnormal{TV}(\mu(\omega),\mu'(\omega))$. 
\begin{lemma}\label{lem:TV}
Under Assumption \ref{asm:growth}, there exists $C \in (0,\infty)$ depending on  $\varphi$ such that 
\begin{equation}
    |v(\mu)-v(\mu')| \le C \ \lVert \textnormal{TV}(\mu,\mu') \rVert_{L^2(\Q)}, \quad \forall \ \mu,\mu' \in \underline{\vartheta}(\calT). 
\end{equation}
In particular, 
$|v(\tau)-v(\tau')| \le C \Q(\tau \ne \tau')^{1/2}$  for any stopping times $\tau, \tau' \in  \vartheta(\calT)$.
\end{lemma}
\begin{proof}
      Define  $\Phi =  \sum_{t\in \calT} |\varphi(t,X_{t})| $ and  
      recall from Assumption \ref{asm:growth} that   $\varphi(t,\cdot) \in  \calL_a$ for some $a>0$. Since  $\phi 
    \in \calL_{a}$ implies $\phi^2 
    \in \calL_{2a}$, we obtain for some constant $C\in (0,\infty)$ that   
$$\E[\Phi^2] \le |\calT|  \sum_{t\in \calT} \E^{\Q}[\varphi(t,X_{t})^2] =   |\calT|  \sum_{t\in \calT} \underbrace{K_{0,t}\varphi^2(t,x_0)}_{\in \  \calL_{2a}} \le  C\sum_{t\in \calT}  [1+|x_0|^{2a}] < \infty. $$ 
Thus  $\Phi \in L^2(\Q)$. Next, observe that 
\begin{align*}
    |v(\mu) - v(\mu')| \le \E^{\Q}\left[  \int_0^T |\varphi(t,X_{t})|\  |\mu-\mu'|(dt) \right]
    \le 2 \E^{\Q}\left[  \Phi\  \text{TV}(\mu,\mu') \right].
\end{align*}
As $\text{TV}(\mu,\mu') \le 1 \in L^2(\Q)$, the result follows from the Cauchy-Schwarz inequality and  setting $C =  2 \lVert\Phi \rVert_{L^{2}(\Q)}$. For the second claim, it suffices to observe that  $\text{TV}(\delta_{\tau},\delta_{\tau'}) = \mathds{1}_{\{\tau \ne \tau'\}}$. 
\end{proof}

\begin{remark} Set $\overline{\textnormal{TV}}_2(\mu,\mu') := \lVert \textnormal{TV}(\mu,\mu') \lVert_{L^2(\Q)}$ for relaxed stopping rules $\mu,\mu' \in  \underline{\vartheta}(\calT)$. Then,  \cref{lem:TV} states that $ \mu \mapsto v(\mu)$ is Lipschitz continuous with respect to $\overline{\textnormal{TV}}_2$. 
\end{remark}

%% file: Continuity.tex
\section{Continuity of the value functional}\label{sec:continuity}
First, we {\rr{state}} a distributional assumption on  the process  $\alpha(X) = (\alpha(X_t))_{t\in \calT}$. 
\begin{asm}
\label{asm:absCont}
For every $t\in \calT$, the conditional law of $\alpha(X_t)$  given $\Xi(X_t)$
is absolutely continuous with respect to the Lebesgue measure,
and admits  a  modulus of continuity $\rho_t:[0,\infty) \to [0,1]$ satisfying,
$$
 \Q(\alpha(X_t)\in [a,a+\iota] \ \large{|}\ \Xi(X_t)) \le \rho_t(\iota), \qquad \Q-a.s, \; \forall a, \iota \ge 0.
 $$
\end{asm}
\begin{example}  Let us verify  Assumption \ref{asm:absCont} for max-call options on $d=2$ assets seen in \cref{ex:maxcall}, namely $\alpha(x) = x_1\vee x_2$ and $\Xi(x) = \frac{x}{\alpha(x)}$. We also suppose that  the  law of $X_t = (X_t^1,X_t^2)$ has no atoms and denote the joint density of $X_t$ by $\phi_t$. If  $\Xi(X_t) \in \calE_b:= \{(1,\xi_2) : \xi_2 \le b \}$, then $\alpha(X_t) = X_t^1$ so that 
\begin{align*}
    \Q(\alpha(X_t)\le a \ \large{|}\ \Xi(X_t) \in \calE_b ) = \Q(X_t^1\le a \ \large{|}\ X_t^2 \le b X_t^1 ) 
    = \frac{1}{\Q(X_t^2 \le b X_t^1)} \int_0^a\int_{0}^{bx_1} \phi_t(x_1,x_2)dx_2dx_1.
\end{align*}
Hence  for all $b\in \R$ such that $\Q(X_t^2 \le b X_t^1)>0$, we have  $\Q(\alpha(X_t)\in [a,a+\iota] \ \large{|}\ \Xi(X_t) \in  \calE_b) \le \rho_t(\iota)$ as claimed. 
\end{example}
We are now ready to  state our first  continuity result. Throughout, we adopt the notation $\calK \subset \subset \calE$ if $\calK$ is a compact subset of $\calE$. 
\begin{theorem}\label{thm:unifcontV}
   Let  $C \in (0,\infty)$ be as in  \cref{lem:TV} and fix $\delta > 0$. Then,
    under Assumptions \ref{asm:absCont},  
    there exists a compact set $\calK \subset \subset \calE$  such that for all $\iota >0$ and  $f,f'\in \frakF$, 
    \begin{equation}\label{eq:unifcontV}
        \lVert f - f' \rVert_{L^{\infty}(\calT \times \calK)}
        \le \iota \ \Longrightarrow \ |v(\tau_{f}) - v(\tau_{f'})|\le C\sqrt{\rho (\iota) + \delta},
    \end{equation}
where $\rho_t$ is the modulus of continuity of 
$a \mapsto \Q(\alpha(X_t) \in [0,a] \ | \ \Xi(X_t))$  and
$\rho(\iota) = \sum_{t\in\calT} \rho_t(\iota) $.  If $\calE$ is compact, 
then $f \mapsto v(\tau_f)$ is uniformly continuous in  the supremum distance.
\end{theorem}

\begin{remark}\label{rem:bdedDensity}
If Assumption \ref{asm:absCont} holds and for all $t\in \calT$, 
 the density $\phi_t(x)dx := \Q(\alpha(X_t)\in dx)$  is  bounded,  then the cumulative distribution function of $\alpha(X_t)$ is Lipschitz continuous. In particular, its modulus of continuity satisfies $\rho_t(\iota) \le  \lVert \phi_t \rVert_{L^{\infty}(\R)} \iota$. Defining $c = \sum_{t\in \calT}\lVert \phi_t \rVert_{L^{\infty}(\R)}$, we can thus rewrite the right side of \eqref{eq:unifcontV}   as 
    \begin{equation} \label{eq:valueLip}
       \lVert f - f' \rVert_{L^{\infty}(\calT \times \calK)}\le \iota \; \Longrightarrow \; |v(\tau_{f}) - v(\tau_{f'})|\le C\sqrt{c \ \iota + \delta}. 
    \end{equation}
\end{remark}

\begin{figure}[t]
    \centering
    \caption{Illustration of  
    the symmetric difference between the stopping regions $\frakS_t(f)$ and  $\frakS_t(f')$ (purple region) contained in the interval $[f \wedge f',f \wedge f'+\iota)$ (orange region); see proof of \cref{thm:unifcontV}. The stopping decisions with respect to $f$ and $f'$ coincide in the red (stop) and blue (continue) region. }
    \vspace{-1mm}
    \includegraphics[width = 2.7in,height = 2.0in]{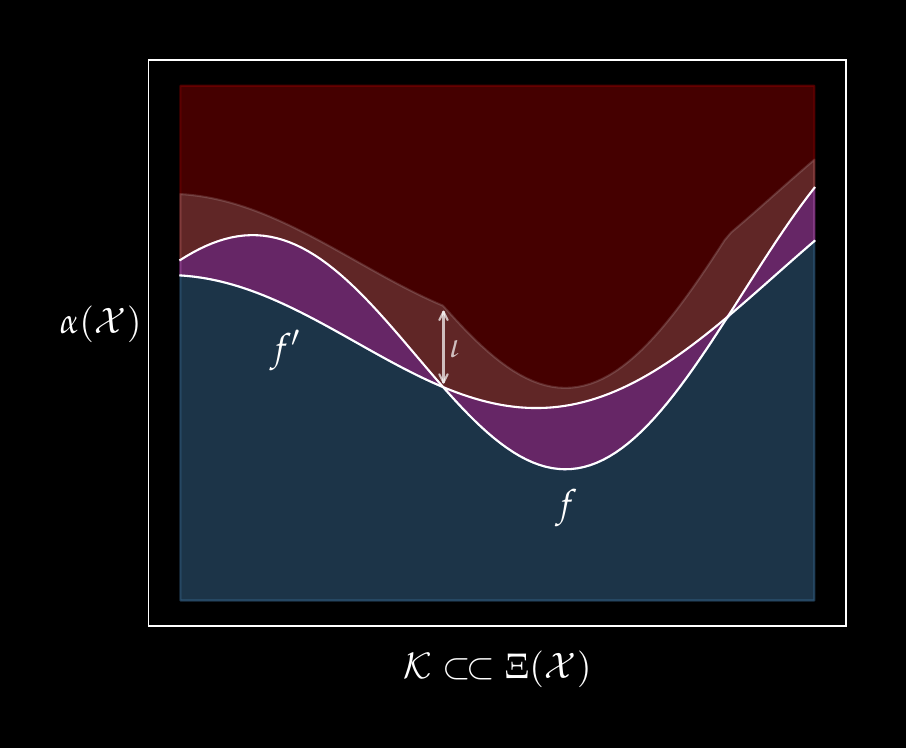}
    
    \label{fig:symdiff}
\end{figure}

\begin{proof} (\cref{thm:convCont}) \underline{\textit{Step ${1}$.}} 
 Since $\calE = \Xi(\calX)$   is locally compact (being the continuous image of a Euclidean space), 
 every finite Borel measure on it 
 is tight. 
Hence,  there exists $\calK \subset \subset \calE$ such that 
 \begin{equation}\label{eq:OmegaK}
   \Q(\Omega_{\calK}) \ge 1-\delta, \quad \Omega_{\calK} := \{\omega \in \Omega \ : \ \Xi(X_t(\omega)) \in \calK  \; \ \forall t\in \calT\}.
\end{equation} 
 \underline{\textit{Step ${2}$.}} Fix $\iota>0$ and $f,f'\in \frakF$ such that   $\lVert f - f' \rVert_{L^{\infty}(\calT \times \calK)}\le \iota$. Using \cref{lem:TV}, then 
$$|v(\tau_{f})-v(\tau_{f'})| \le C \Q(\tau_{f} \ne \tau_{f'})^{1/2},$$
so it remains to find the appropriate upper bound for $\Q(\tau_{f} \ne \tau_{f'})$.   
      First, note that the event $  \{\tau_{f} \ne \tau_{f'}\}$ is contained in   
      $$\tilde{\Omega} = \{\omega \in \Omega  : X_t(\omega) \in \frakS_t(f) \triangle \frakS_t(f')   \ \text{ for some } t\in \calT\}.  $$   
Indeed, for $\omega \notin \tilde{\Omega}$,
the sets $\{t \in \calT  :  X_t(\omega) \in  \frakS_t(f) \}$ and $\{t \in \calT  : X_t(\omega) \in  \frakS_t(f') \}$  are identical. Hence $\tau_{f}(\omega) = \tau_{f'}(\omega)$, which shows that $\Omega \setminus \tilde{\Omega} \subseteq \{\tau_f = \tau_{f'}\}$. 
Next, notice that 
\begin{align*}
     x &\in \frakS_t(f) \triangle \frakS_t(f')   \, \Longleftrightarrow \alpha(x) \in I(t,\Xi(x)), \quad x \in \calX, 
\end{align*}
with the interval $I(t,\xi) := [(f \wedge f')(t,\xi), (f \vee f')(t,\xi))$.   
Since $\lVert f - f' \rVert_{L^{\infty}(\calT \times \calK)}\le \iota$, then  
 $I(t,\xi)$ has  length at most  $\iota$ 
 for every $(t,\xi) \in \calT \times \calK$. See \cref{fig:symdiff} for an illustration.   
Using Assumption \ref{asm:absCont} and recalling that 
$\rho_t$ is the modulus of continuity of $a \mapsto \Q(\alpha(X_t) \in [0,a] \ | \  \Xi(X_t))$, then  
\begin{equation}\label{eq:absCont}
   \Q(\alpha(X_t) \in I(t,\Xi(X_t)) )  \le  \rho_t(\iota).
\end{equation}
We conclude from \eqref{eq:absCont} and 
the union bound  that 
\begin{align*}
    \Q( \tilde{\Omega} \cap   \Omega_{\calK})
    \le \sum_{t\in \calT} \Q(\{\alpha(X_t)\in I(t,\Xi(X_t))\} \cap \Omega_{\calK}  )
    \le \sum_{t\in \calT} \rho_t(\iota) 
    &=  \rho(\iota).
\end{align*}
Combining Step 1 with the above observations, we finally  obtain
\begin{equation*}
  \Q(\tau_{f} \ne  \tau_{f'}) \le   \Q(\{\tau_{f} \ne  \tau_{f'}\} \cap \Omega_{\calK})  + \Q(\Omega \setminus \Omega_{\calK}) \le  \Q( \tilde{\Omega} \cap   \Omega_{\calK})  + \delta 
    \le \rho(\iota) + \delta . 
\end{equation*}  
\end{proof} 

We have the following corollary of \cref{thm:unifcontV}, already stated in \eqref{eq:relaxConv}, proving the  convergence of the values obtained by the relaxed stopping rules to the strict ones. 
\begin{corollary}
    \label{cor:relaxUnif}
Under Assumptions \ref{asm:growth} and \ref{asm:absCont}, for all $\varepsilon >0$ and  $f\in \frakF$, 
    \begin{equation}\label{eq:relaxUnif}
   |v(\mu_{f}^{\varepsilon}) - v(\tau_{f})| \le C\sqrt{\rho (\varepsilon)} ,
\end{equation}
    with $\rho, C$ given in \cref{thm:unifcontV,lem:TV}, respectively. 
\end{corollary}

\begin{remark}
Under the additional assumption given in \cref{rem:bdedDensity}, equation  \eqref{eq:relaxUnif} becomes, 
    \begin{equation} \label{eq:relaxUnif2}
        |v(\mu_{f}^{\varepsilon}) - v(\tau_{f})| \le \tilde{C}\sqrt{\varepsilon}, 
    \end{equation}
    where $\tilde{C} = C c$ and $c$ defined in \cref{rem:bdedDensity}. We therefore obtain a rate of convergence equal to $1/2$ for the value associated to $\mu_f^{\varepsilon}$ as $\varepsilon \to 0$. 
    \end{remark}

 \begin{proof} (\cref{cor:relaxUnif}) 
  In view of \cref{lem:TV},  $|v(\mu_{f}^{\varepsilon}) - v(\tau_{f})| \le C  \| \textnormal{TV}(\mu_{f}^{\varepsilon},\delta_{\tau_f}) \lVert_{L^2(\Q)}.$ 
  Also by the definition of $\mu_{f}^{\varepsilon}$ in  \eqref{eq:relaxedRule} and the fact that  $\delta_{\tau_f}$ is a (random)  Dirac measure,  the total variation distance simplifies to 
\begin{align}
     \text{TV}(\mu_{f}^{\varepsilon},\delta_{\tau_f}) &= \frac{1}{2} \Big[ 1 - P_{f,\tau_{f}}^{\varepsilon} +  \sum_{t\in \calT\setminus\{\tau_{f}\}} P_{f,t}^{\varepsilon} \Big] = 1 - P_{f,\tau_{f}}^{\varepsilon} \le \mathds{1}_{\{\mu_{f}^{\varepsilon} \ \ne \ \delta_{\tau_{f}}\}} \quad \Q-\text{a.s.} 
\end{align}
     We now claim that  $\{\tau_{(f-\varepsilon)} = \tau_{f}\} \subseteq  \{\mu_{f}^{\varepsilon} = \delta_{\tau_{f}}\}$. Indeed,  if  $\tau_{(f-\varepsilon)}(\omega) = \tau_{f}(\omega)$,  the discrete path $(X_{t}(\omega))_{t\in \calT}$ never enters the fuzzy region  $\frakB_t^{\varepsilon}(f) = \frakS_t(f-\varepsilon) \setminus \frakS_t(f)$ prior to $\tau_f(\omega)$. Thus, the relaxed stopping rule 
     $\mu_f^{\varepsilon}(\omega)$ coincide with $\tau_f(\omega)$. Consequently, 
     $$  
     \lVert \textnormal{TV}(\mu_{f}^{\varepsilon},\delta_{\tau_f}) \lVert_{L^2(\Q)}^2  \le \Q(  \mu_{f}^{\varepsilon} \ \ne \ \delta_{\tau_{f}}) \le   \Q(  \tau_{(f-\varepsilon)} \ne \tau_{f}  ). 
     $$
  We now follow Step 2 in the proof of \cref{thm:unifcontV}  with  $f' = f-\varepsilon$, $\iota =\varepsilon$ to conclude that  $\Q(  \tau_{f-\varepsilon} \ne \tau_{f}  ) \le \rho(\varepsilon)$. We note that as $\lVert f-f'\rVert_{L^{\infty}(\calT\times \calE)} = \varepsilon$, no localization is needed here as done in the proof of \cref{thm:unifcontV}.  
 \end{proof}

%% file: convSequential.tex
\subsection{Relaxed $L^{\infty}$ metric and second continuity result} 
\label{sec:relaxLinfty}

The continuity result in \cref{thm:unifcontV} only applies when boundaries are close in the $L^{\infty}$ sense. 
Although reasonable for continuous boundaries, this condition is typically too restrictive 
for semicontinuous boundaries. We therefore introduce a weaker topology {\rr{induced by the following metric.
\begin{definition}
\label{def:relaxedSupDist} 
Given $f,g: \calE\to \R$, the \textit{relaxed $L^{\infty}$ metric}  is defined as 
$\frakm(f,g):= \frakm_d(f,g)+\frakm_d(g,f)$, 
where 
$$
\frakm_d(f,g): =  \inf_{\psi\in \Psi}  
\sup_{\xi \in \calE} \ \big[(f(\psi(\xi))- g(\xi))^+ +  |\psi(\xi)-\xi| \big],
$$
and $\Psi$ is the set of all endofunctions $\psi:\calE \to \calE$. 
For  $f,g \in \frakF$, we also write  
\begin{equation}\label{eq:relaxDistF}
\frakm(f,g)  = \max_{t \in \calT}  \frakm( f(t,\cdot), g(t,\cdot)). 
\end{equation}
We denote the relaxed $L^\infty$ distance on any compact subset
$\calK$ by $\frakm_{\calK}(f,g) := \frakm(f\mathds{1}_{\calK} , g\mathds{1}_{\calK})$.
\end{definition}
We note that the epigraph, $\text{epi}(f)$, of a lower semi-continuous function $f$
is closed and for any $\xi \in \calE$,  $\inf_{\psi\in \Psi}  \big[(f(\psi(\xi))- g(\xi))^+ +  |\psi(\xi)-\xi| \big]$
is equal to the distance from $(\xi,g(\xi))$  to $\text{epi}(f)$.}}

{\rr{The relaxed $L^{\infty}$ metric comes as a natural notion for  semicontinuous functions 
and is closely related to \textit{inf convolutions} discussed in \cref{sec:convLSC}. 
The transformations  in $\Psi$ echo the  reparametrization maps appearing in the  
Skorokhod $J1$ distance for càdlàg processes. 
A notable difference is that unlike 
the Skorokhod metric,  $\psi$ needs not be invertible. Additionally,
the \emph{``directed"} function 
$\frakm_d$ is an asymmetric metric studied in \cite{mennucci}
as we prove below.}}
{\rr{ \begin{lemma}
 \label{lem:metric}
 The function $\frakm_d$ is an 
 asymmetric metric, i.e., it has the following properties:
 \begin{itemize}
\item $\frakm_d(f,g)=0$ if and only if $f\le g$ on $\calE$.
\item (triangle inequality) 
$\frakm_d(f,h) \le \frakm_d(f,g) +\frakm_d(g,h)$, for any $f,g,h \in \frakF$.
\end{itemize}
Consequently, $\frakm$ is a metric on $\frakF$ and also  dominated by the supremum distance.
\end{lemma}
\begin{proof}
First statement follows directly from the definition.
Fix $f,g,h \in \frakF$.  For $\psi,\phi \in \Psi$, set $\varphi := \phi \circ \psi$.
Then, $\varphi \in \Psi$ and we have
\begin{align*}
\frakm_d(f,h) 
&\le   \sup_{\xi \in \calE} \ \big[(f(\varphi(\xi))-h(\xi))^+ +  |\varphi(\xi)-\xi| \big]\\
&\le \sup_{\xi \in \calE} \ \big[(f(\phi(\psi(\xi))- g(\psi(\xi)))^+  +(g(\psi(\xi)) -h(\xi) )^+
+  |\psi(\xi)-\xi| + |\phi(\psi(\xi))-\psi(\xi) | \big]\\
&\le \sup_{\xi \in \calE} \ \big[(f(\phi(\psi(\xi))- g(\psi(\xi)))^+ + |\phi(\psi(\xi))-\psi(\xi) |   \big] +
\sup_{\xi \in \calE} \big[(g(\psi(\xi)) -h(\xi) )^+ +  |\psi(\xi)-\xi|\big] \\
&\le \sup_{\xi' \in \calE} \ \big[(f(\phi(\xi'))- g(\xi'))^+  +  |\phi(\xi') -\xi'|  \big]  +
\sup_{\xi \in \calE} \big[ (g(\psi(\xi)) - h(\xi))^+  + |\psi(\xi)-\xi|\big] .
\end{align*}
Since above holds for every  $\psi,\phi \in \Psi$, 
we conclude that
\begin{align*}
\frakm_d(f,h) 
&\le \inf_{\phi \in \Psi}  \ \sup_{\xi' \in \calE} \ \big[(f(\phi(\xi'))- g(\xi'))^+  +  |\phi(\xi') -\xi'|  \big] +
\inf_{\psi \in \Psi} \sup_{\xi \in \calE} \ \big[ (g(\psi(\xi)) - h(\xi))^+  + |\psi(\xi)-\xi|\big] \\
& = \frakm_d(f,g) +\frakm_d(g,h).
\end{align*}
Then, the symmetrized function $\frakm$ is a metric.
\end{proof}}}

Next, we adapt \cref{thm:unifcontV} to obtain a continuity type result of 
$f\mapsto v(\tau_f)$ with respect to  the relaxed $L^{\infty}$ metric. 

\begin{theorem}\label{thm:lscV}
Let $C \in (0,\infty)$ be as in \cref{lem:TV} and fix $\delta > 0$.
Then, under Assumptions \ref{asm:absCont},  
there exists 
$\calK \subset \subset \calE$  such that for all
$f \in \frakF$ and a {\rr{sequence $f_n \le f$ in  $\frakF$}}, we have 
$$
\lim_{n\to \infty }\frakm_{\calK}(f,f_n) = 0 \; \; \Longrightarrow \; \; 
\limsup_{n\to \infty }\ |v(\tau_{f}) - v(\tau_{f_n})|\le C\sqrt{\delta}. 
$$
\end{theorem}

\begin{proof}  
{\rr{As  in the proof of  \cref{thm:unifcontV}, there exists   $\calK \subset \subset \calE$  such that 
 \begin{equation*}
\Q(\Omega_{\calK}) \ge 1-\delta, \quad \Omega_{\calK} 
= \{\omega \in \Omega \ : \ \Xi(X_t(\omega)) \in \calK  \; \ \forall t\in \calT\}.
\end{equation*}
Consider a sequence  $f_n \le f$ in  $\frakF$, such that 
$\frakm_{\calK}(f,f_n)$ converges to zero.  Then, there are $\psi_n \in \Psi$ satisfying,
$$
c_n:= \sup_{\xi \in \calK} \ \big[( f(\psi_n(\xi))-f_n(\xi))^++  |\psi_n (\xi)-\xi| \big] \longrightarrow 0, \quad n \to \infty.
$$
Since by \eqref{eq:frakF} $f$ is lower semi-continuous and $c_n$ tends to zero, for every $\xi \in \calK$,
\begin{align}
\nonumber
0 &\le \limsup_{n \to \infty} f(\xi)- f_n(\xi) \le
\limsup_{n \to \infty} \ (f(\xi)- f(\psi_n(\xi))) + \limsup_{n \to \infty} \ (f(\psi_n(\xi)) - f_n(\xi))\\
\label{eqn:converge}
&\le  f(\xi)-\liminf_{n \to \infty} f(\psi_n(\xi)) + \limsup_{n \to \infty} \ (f(\psi_n(\xi)) - f_n(\xi))^+
\le 0.
\end{align}
Following the proof of \cref{thm:unifcontV}, we obtain,  
\begin{align*}
\{\tau_{f} \ne \tau_{f_n}\} \subseteq  \Omega_n : = 
\bigcup_{t\in \calT}\big\{ X_t\in \frakS_t(f) \triangle \frakS_t(f_n) \big\} = 
\bigcup_{t\in \calT}\big\{ f_n(t,\Xi(X_t)) \le \alpha(X_t) <  f(t,\Xi(X_t))  \big\}.
\end{align*}
Set $\chi_n$ to be the characteristic function of the set
$\Omega_n \cap \Omega_{\calK}$.  
In view of \eqref{eqn:converge}, $\chi_n$ converges to zero.}}
{\rr{Then, from the bounded convergence theorem, we have 
$
\limsup_{n\to \infty }\ \Q(\Omega_n \cap \Omega_{\calK})  = 0.
$ 
Following   the same rationale as in the proof of  \cref{thm:unifcontV}, we conclude  that 
 \begin{align*}
    \limsup_{n\to \infty }\ |v(\tau_{f})-v(\tau_{f_n})| &\le  \limsup_{n\to \infty }\  C\left[\Q(\Omega_n \cap \Omega_{\calK}) + \delta\right]^{1/2} = C\sqrt{\delta}. 
\end{align*}}}
\end{proof}

%% file: Algorithm.tex
\section{Application to neural  stopping boundaries} \label{sec:NOSB}

\rr{In this section, we apply the continuity results to prove the convergence 
of the \textit{Neural Optimal Stopping Boundary method}   developed in
\cite{ReppenSonerTissotFB}.  We outline the algorithm for the convenience of the reader.} Under a suitable coordinate system  $(\alpha,\Xi)$ as in Assumption \ref{asm:star}, the algorithm looks for a \textit{neural stopping boundary}, namely a feedforward neural network $g^{\theta}: \calT \times \calE \to \R_+ $  in a (sufficiently large) hypothesis space $\frakN$ 
such that $$v(\tau_{g^{\theta}}) \approx \sup_{f\in \frakF}
v(\tau_f) \  = v^{\diamond}.$$ 
Additional details on the  architecture and  size of the neural networks 
for this problem are given in \cite{ReppenSonerTissotFB}.
The parameter vector $\theta$ is thereafter trained using stochastic gradient ascent. As explained at the beginning of \cref{sec:relaxation}, the  stopping rule needs to be relaxed to avoid vanishing gradient issues. 
For some fixed fuzzy region width $\varepsilon>0$, we therefore 
aim to 
\begin{equation}
\label{eq:OSNNRelax}
\text{maximize}
\ \ 
v(\mu_g^{\varepsilon}) =  \E^{\Q}\left[ \int_{\calT}\varphi(t,X_{t})\mu_g^{\varepsilon}(dt)\right]
\qquad
\text{over}\  g\in \frakN,
\end{equation} 
with the relaxed stopping rule $\mu_g^{\varepsilon}$ given in \eqref{eq:relaxedRule}.  
The vanishing gradient problem is henceforth resolved and owing to well-known convergence results in stochastic optimization \cite{RobbinsMonro,RobbinsSiegmund}, we can find for every $\gamma > 0$ 
an 
$\gamma-$\textit{maximizer} 
$g^{\varepsilon,\gamma} \in \frakN$, that is 
\begin{equation}\label{eq:fuzzytraining}
   v(\mu^{\varepsilon}_{g^{\varepsilon,\gamma}})\ge \sup_{g \in \frakN} \  v(\mu^{\varepsilon}_g) - \gamma.  
\end{equation} 
Finally, the value of the optimal stopping problem is computed using  the strict stopping rule $\tau_{g^{\varepsilon,\gamma}}$ instead of  $\mu^{\varepsilon}_{g^{\varepsilon,\gamma}}$. 
The method is summarized in  \cref{alg:NOSB}. 
For further details, see Appendix C in \cite{ReppenSonerTissotFB}.  The goal of this section is to show that $
   v(\tau_{g^{\varepsilon,\gamma}})$ can be made arbitrarily close to the optimal value $ v^{\diamond} 
$ when   $\varepsilon, \gamma$ are small enough. 

\begin{remark}\label{rem:MC}
When computing the value associated to the trained stopping rule, say $\tau_{g}$,  an additional inaccuracy comes  from the Monte Carlo simulations of the process; see step III in \cref{alg:NOSB}. There is nevertheless a clear  control of this error.  Indeed, it is classical that 
$$|v(\tau_{g}) - \hat{v}_J(\tau_{g})| = \calO(J^{-1/2}),$$
 when the random reward $\varphi(\tau_{g},X_{\tau_{g}})$ is  square integrable. Under Assumption \ref{asm:absCont}, this is indeed the case (see also the proof of \cref{lem:TV}). 
\end{remark}

\begin{remark}
In \cite{ReppenSonerTissotFB},  a two-sided fuzzy region is considered, namely 
  $ \{ x\in \calX \ : \  |d(t,x;f)| < \varepsilon \},$ 
where $d(t,\cdot;f)$ is the  \textit{signed distance function} to  $\partial \frakS_t(f)$. 
Although the two representations are equivalent, one-sided fuzzy regions are  more practical towards a convergence analysis.
\end{remark}

\begin{algorithm}[t]
\caption{Neural Optimal Stopping Boundary Algorithm (summary) }\label{alg:NOSB}

\vspace{1mm}
\textbf{Given}: $\varphi=$ payoff, $ \calT =$ exercise dates, $\Theta=$ parameter set, $I =  \#$ training iterations,  $\zeta=$ learning rate process,  $\varepsilon=$ fuzzy region width, $B = $ batch size, $J= \#$ simulations for the value  \\[-0.8em]  \hrule

\begin{itemize}
 \setlength \itemsep{-0.1em}
\item[I.] \textbf{Initialize} $\theta_0 \in \Theta$
\item[II.] For $i = 0,\ldots,I-1$:
    \begin{enumerate}
    \setlength \itemsep{0.5em}
        \item \textbf{Simulate} $X(\omega_b) = \{X_t(\omega_b)\}_{t\in \calT}$, $\omega_b \in \Omega $, $b=1,\ldots, B$
        \item \textbf{Compute}  $ \hat{v}_B(\mu^{\varepsilon}_{g^{\theta_i}}) :=\frac{1}{B}\sum\limits_{b=1}^B \int_{\calT}  \varphi(t,X_{t}(\omega_b)) \ \mu^{\varepsilon}_{g^{\theta_i}}(\omega_b)(dt) $ 
\item \textbf{Gradient step}: $\theta_{i+1} = \theta_{i} + \zeta_i \nabla_\theta \hat{v}_B(\mu^{\varepsilon}_{g^{\theta_i}}) $
   \end{enumerate}
\item[III.] \textbf{Return} $g = g^{\theta_{I}}$ and $\hat{v}_J(\tau_{g}) := \frac{1}{J}\sum\limits_{j=1}^J \varphi(\tau_{g}(\omega_j),X_{\tau_{g}}(\omega_j)), \  \omega_j \in \Omega $, $\ j=1,\ldots, J$
\end{itemize}
\vspace{-3mm}
\end{algorithm}

\subsection{Convergence analysis: continuous boundaries}

We first outline our strategy to prove convergence of the Neural Optimal Stopping Boundary method when the optimal stopping boundary $f^{\diamond}$ is continuous. First, the hypothesis space $\frakN$ (neural networks in this case) must satisfy a universal approximation property; see Assumption \ref{asm:UAP} below. Thus, there exists $g^{\diamond} \in \frakN$ arbitrarily close to $f^{\diamond}$ in compact sets. We can  therefore invoke our first continuity result  (\cref{thm:unifcontV}) to ensure that the values of the stopping strategies $\tau_{f^{\diamond}}$, $\tau_{g^{\diamond}}$ are close as well. Next, we relax the problem and use \cref{cor:relaxUnif}  to guarantee that $v(\mu^{\varepsilon}_{g^{\diamond}}) \approx  v(\tau_{g^{\diamond}})$ for small fuzzy parameter $\varepsilon$. We then train the neural boundary  
to construct $g^{\varepsilon,\gamma} \in \frakN$ such that $v(\mu^{\varepsilon}_{g^{\diamond}}) \approx v(\mu^{\varepsilon}_{g^{\varepsilon,\gamma}})$ and finally  conclude that $v(\mu_{g^{\varepsilon,\gamma}}^{\varepsilon}) \approx  v(\tau_{g^{\varepsilon,\gamma}})$  in view of   \cref{cor:relaxUnif}.  The above steps can be summarized as follows: 
\begin{align*}
    v^{\diamond} &= v(\tau_{f^{\diamond}}) && \\  
&\approx v(\tau_{g^{\diamond}}) && \longleftarrow \ \textnormal{approximation and continuity (Assumption \ref{asm:UAP} and \cref{thm:unifcontV})}\\ 
&\approx  v(\mu_{g^{\diamond}}^{\varepsilon}) && \longleftarrow \ \textnormal{relaxation (\cref{cor:relaxUnif})} \\ 
&\approx  v(\mu_{g^{\varepsilon,\gamma}}^{\varepsilon}) && \longleftarrow \ \textnormal{stochastic optimization  \cite{RobbinsMonro,RobbinsSiegmund}} \\ 
&\approx v(\tau_{g^{\varepsilon,\gamma}}). && \longleftarrow \ \textnormal{relaxation (\cref{cor:relaxUnif})}
\end{align*}
We start with the approximation of $f^{\diamond}$. Let $\frakF_c  = \{f \in \frakF :  f(t,\cdot) \in \calC(\calE) \; \forall t\in \calT\}$ denotes the set of continuous boundaries and consider  a  hypothesis space $\frakN \subset \frakF_c$ satisfying the following  \textit{universal approximation property}. 
\begin{asm}\label{asm:UAP}
For all 
$f\in \frakF_c$, compact set  $\calK \subset \subset \calE$  and $\iota >0$, there exists  $g\in \frakN$ such that  
\begin{equation}\label{eq:UAP}
    \lVert f - g \rVert_{L^{\infty}(\calT \times \calK )} \le \iota. 
\end{equation}
\end{asm}

We refer to elements in a family $\frakN \subset \frakF_c$ such that 
Assumption \ref{asm:UAP} holds as \textit{universal approximators}. The Stone-Weierstrass theorem gives an example of  universal approximators, namely the polynomials. This class, however, is  numerically unsuitable when the Hausdorff dimension of $\calE$ is large. In our case, $\frakN$ is evidently the family of all feedforward neural networks $g^\theta:\calT\times \  \calE \to \R_+$ with finite-dimensional parameter vector $\theta$. 
Assumption \ref{asm:UAP} follows classically from the 
 Universal Approximation Theorem (UAT) for neural networks \cite{Cybenko,Hornik}. 

In this section, we assume that $f^{\diamond}(t,\xi)$ is  continuous in $\xi$, i.e. $f^{\diamond} \in \frakF_c$. We can therefore directly use $f^{\diamond}$ in Assumption \ref{asm:UAP}. Together with the results from \cref{sec:continuity}, the next convergence theorem follows naturally. 
\begin{theorem}\label{thm:convCont}
   Consider $\frakN \subset \frakF_c$  satisfying Assumption \ref{asm:UAP} 
   and suppose that Assumption  \ref{asm:absCont} holds.  
   Further assume that the 
   optimal stopping boundary is continuous, i.e., $f^{\diamond}\in \frakF_c$.
   Then, for every $\delta >0$,   there exists $\varepsilon, \gamma>0$  such that any $g^{\varepsilon,\gamma} \in \frakN$ with the property 
   $v(\mu^{\varepsilon}_{g^{\varepsilon,\gamma}}) \ge \sup_{g \in \frakN} \  v(\mu^{\varepsilon}_g) - \gamma$ 
   also satisfies,
   \begin{equation*}
    v(\tau_{g^{\varepsilon,\gamma}})\ge v^{\diamond} - \delta 
    =v(\tau_{f^\diamond}) -\delta
    = \sup_{f \in \frakF} v(\tau_f) -\delta.  
\end{equation*} 
\end{theorem}

\begin{proof} \underline{\textit{Step ${1}$.}}  Let $\delta'>0 $ be arbitrary. From  \cref{thm:unifcontV}, there exist $\delta''>0$, $\calK\subseteq \calE$ compact,  and $\iota>0$ such that
\begin{equation}
    \lVert f - f' \rVert_{L^{\infty}(\calT \times \calK)}\le \iota \ \Longrightarrow \ |v(\tau_{f}) - v(\tau_{f'})|\le C\sqrt{\rho(\iota) + \delta''}\le \delta'. 
\end{equation}
Invoking Assumption \ref{asm:UAP}, there exists   $g^{\diamond} \in \frakN$ such that 
$\lVert f^{\diamond} - g^{\diamond} \rVert_{L^{\infty}(\calT \times \calK)}\le \iota. $
We conclude  from \cref{thm:unifcontV}  that $v(\tau_{g^{\diamond}}) \ge v^{\diamond} - \delta'$. \\[-1em]

   \underline{\textit{Step ${2}$.}} 
   For $\varepsilon >0$, set
   $v^{\varepsilon} :=\sup_{g \in \frakN} v(\mu^{\varepsilon}_g)$.  
   For any  $f \in \frakF$, \cref{cor:relaxUnif} implies that 
    $v(\mu^{\varepsilon}_f) \to v(\tau_f)$  as $\varepsilon \to 0$. 
    In particular, taking $f=g^{\diamond}$ from Step 1 yields 
    $$\liminf_{\varepsilon \to 0 } v^{\varepsilon} \ge  \liminf_{\varepsilon \to 0 } v(\mu_{g^{\diamond}}^{\varepsilon}) = v(\tau_{g^{\diamond}}) . $$
    Thus, one can pick $ \bar{\varepsilon}  >0$ such that $v^{\varepsilon} \ge v(\tau_{g^{\diamond}})  - \delta'$ for all $\varepsilon \in (0,\bar{\varepsilon})$. \\[-1em]
    
     \underline{\textit{Step ${3}$.}}  Using again \cref{cor:relaxUnif}, there exists  $ \ \varepsilon \in (0,\bar{\varepsilon})$ such that for all $f \in \frakF$,    \begin{equation}\label{eq:step3}
      |v(\tau_{f}) - v(\mu^{\varepsilon}_{f})| \le C\sqrt{\rho(\varepsilon)} \le  \delta'.
     \end{equation} 
     Consequently, if  $g^{\varepsilon,\gamma} \in \frakN$ is a $\gamma-$maximizer of $v^{\varepsilon}$, then 
     $$v(\tau_{g^{\varepsilon,\gamma}}) \ge v(\mu^{\varepsilon}_{g^{\varepsilon,\gamma}}) - \delta' \ge v^{\varepsilon} - \gamma-\delta'. $$
   \underline{\textit{Step ${4}$.}}
Collecting the above inequalities, we obtain
\begin{align*}
    v(\tau_{g^{\varepsilon,\gamma}})  
    \ge v^{\varepsilon} - \gamma-\delta'  
    \ge  v(\tau_{g^{\diamond}}) - \gamma - 2 \delta'
    \ge v^{\diamond}- \gamma - 3\delta'.
\end{align*}
We can therefore set $\delta' = \gamma = \frac{\delta}{4}$.  
\end{proof}

%% file: AlgorithmLSC.tex
\subsection{Convergence analysis: semicontinuous boundaries}\label{sec:convLSC}

This section consider optimal boundaries $f^{\diamond} = f^{\diamond}(t,\xi)$ that are 
only not necessarily continuous and are only
lower continuous in $\xi$. To project $f^{\diamond}$ onto $\frakN$, we  first carefully construct a sequence of \textit{continuous} boundaries converging to $f^{\diamond}$. These continuous boundaries can thereafter be approximated by elements in $\frakN$. To construct such sequence,  we utilize the notion of \textit{inf/sup convolution},  which is central in the theory of viscosity solutions; see \cite[Section V.5]{FS} and \cite{LasryLions}.   

Let $h:\calE \to \R_+$ be a  lower semicontinuous function and define its \textit{inf convolution}  as
\begin{equation}\label{eq:infconvDef}
      h_\delta(\xi) := \inf_{\xi' \in  \calE}\left[h(\xi') + \frac{1}{\delta} |\xi'-\xi|^2 \right] , \quad \delta >0. 
\end{equation}
    It is easily seen that $h_\delta$ is semiconcave, hence continuous.    
For $f \in \frakF$, the inf convolution of $f$ is understood as $f_{\delta}(t,\cdot) = f(t,\cdot)_{\delta}$   for each $t\in \calT$.  The next proposition shows that $f_{\delta} \uparrow f$ as $\delta \to 0$; see  \cref{fig:infconv}. In particular, the hitting time $\tau_{f_{\delta}}$ of   $\frakS(f_{\delta})$ increases to $\tau_{f}$.    

\begin{figure}[t]
    \centering
    \caption{Inf convolution of $f(t,\cdot)$ for some $t\in \calT$. The monotonicity of $\delta \to f_{\delta}$ implies that the hitting time $\tau_{f_{\delta}}$ 
    increases  to $\tau_{f}$ (pathwise) as $\delta \downarrow 0$.  }
    \includegraphics[width = 2.9in,height = 2.0in]{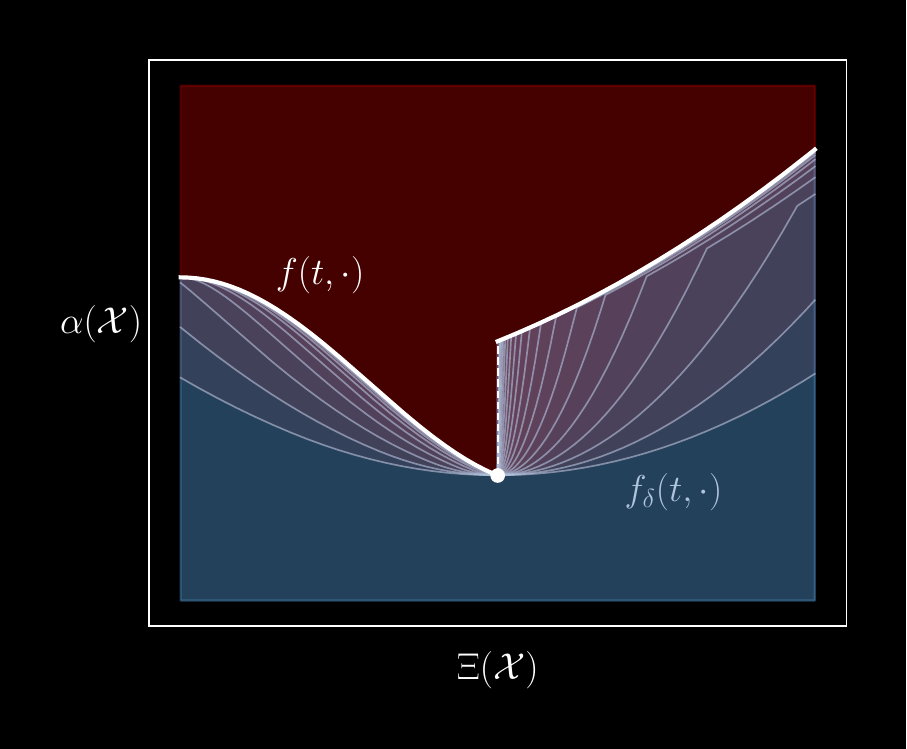}
    \label{fig:infconv}
\end{figure}

\begin{proposition}\label{prop:lsc}
    Let $h: \calE \to \R_+$ be lower semicontinuous and $h_\delta$ be its inf convolution with parameter $\delta$ as in \eqref{eq:infconvDef}. 
    Then  $h_\delta \uparrow h$ pointwise. 
\end{proposition}
\begin{proof}
It is clear that $h_\delta$ increases with $\delta$, hence any sequence $(h_{\delta_n})$ with $ \delta_{l} \downarrow 0$  is non-decreasing. 
Let $\xi_n$ be a $\delta_n-$minimizer of $h_{\delta_n}(\xi)$, i.e. $h(\xi_n) + \frac{1}{\delta_n} |\xi - \xi_n|^2 \le h_{\delta_n}(\xi)+ \delta_n$. Then, 
    $$h(\xi) \ge h_{\delta_n}(\xi) \ge h(\xi_n) + \frac{1}{\delta_n} |\xi - \xi_n|^2 -\delta_n\ge \frac{1}{\delta_n} |\xi - \xi_n|^2 -\delta_n,$$
    which implies that $|\xi - \xi_n|^2 \le \delta_n (\delta_n + h(\xi))$. Therefore $\xi_n \to \xi$, and  the lower semicontinuity of $h$ gives
    $$h(\xi) \le \liminf_{l\to \infty} h(\xi_n)\le \liminf_{l\to \infty} \left[ h(\xi_n) + \frac{1}{\delta_n} |\xi - \xi_n|^2 \right]\le \liminf_{l\to \infty} \left[ h_{\delta_n}(\xi) + \delta_n \right]\le h(\xi). $$
\end{proof}
More can be said when restricting l.s.c. boundaries to a compact set. Accordingly, we define
\begin{equation}\label{eq:infconvCpctDef}
      h_{\delta}^{\calK} (\xi) := \inf_{\xi' \in \calK}\left[h(\xi') + \frac{1}{\delta} |\xi'-\xi|^2 \right], \quad \delta >0, \quad \xi \in \calK \subset \subset \calE. 
\end{equation}
The inf convolution of $f \in \frakF$ in $\calK$ is understood as $f_{\delta}^{\calK}(t,\cdot) = f(t,\cdot)_{\delta}^{\calK}$   for each $t\in \calT$.  
Since $\xi' \mapsto h(\xi') + \frac{1}{\delta} |\xi'-\xi|^2$ is also l.s.c., it attains its minimum on  $\calK$.   
This leads us to the following corollary. Throughout, we write $B_{r}(\xi)$   for the  closed ball  of radius $r\ge 0$ centered at $\xi \in \calE$.   
\begin{corollary}\label{cor:lsc}
    Let  $\calK \subset \subset \calE$ and a lower semicontinuous function $h:\calE \to \R_+$  such that $ c_{\calK} := \lVert h  \rVert_{L^{\infty}(\calK)} < \infty$. Then $\forall \ \delta>0$, $\xi \in \calK$, any minimizer achieving $h_{\delta}^{\calK}(\xi)$ in \eqref{eq:infconvCpctDef} is contained in $B_{\sqrt{\delta c_{\calK}}}(\xi)$. 
\end{corollary}
\begin{proof}
Let $\xi_{\delta} \in \calK$ be any minimizer in \eqref{eq:infconvCpctDef}, i.e. $h_{\delta}^{\calK}(\xi)  = h(\xi_{\delta}) + \frac{1}{\delta} |\xi_{\delta}-\xi|^2$. Then
    $$c_{\calK}\ge h(\xi) \ge h_{\delta}^{\calK}(\xi)  =  h(\xi_{\delta}) + \frac{1}{\delta} |\xi_{\delta}-\xi|^2 \ge \frac{1}{\delta} |\xi_{\delta}-\xi|^2. $$\end{proof}

\begin{remark}
    When the stopping region is the hypograph of the optimal boundary ($\eta =-1$ in \ref{asm:star}) one can employ \textit{sup convolutions} instead, that is
        $$h^{\delta}(\xi) := \sup_{\xi' \in \ \calE}\left[h(\xi') - \frac{1}{\delta} |\xi'-\xi|^2 \right] \ge h(\xi), \quad \delta >0. $$
        We may also define the restriction $h^{\delta}_{\calK}$ to some compact set  $\calK\subset \subset \calE$. This case is simpler since u.s.c functions are automatically bounded above in compact sets. 
\end{remark}

The next lemma states that inf convolutions approximate lower semicontinuous functions arbitrarily well with respect to the relaxed $L^\infty$ distance seen in \cref{def:relaxedSupDist}.  
{\rr{\begin{lemma}\label{lem:infCovApprox}
Let $\calK$ be a compact subset of $\calE$, 
$h:\calE \to \R$  be bounded and lower semi-continuous,
and $h_{\delta}^{\calK}$ be as in \eqref{eq:infconvCpctDef}. 
Then, for any $\iota >0$, there exists $\delta_{\iota, \calK}>0$ such that 
$$
\frakm_{\calK} (h,h_{\delta}^{\calK}) \le \iota,
\qquad \forall\ \delta \in (0, \delta_{\iota, \calK}].
$$
\end{lemma}
\begin{proof}
Fix $\iota > 0$ and $\calK\subset \subset \calE$. 
Set $c_{\calK}:=  \lVert h  \rVert_{L^{\infty}(\calK)}$.
For $\xi \in \calK$, choose $\psi_\delta(\xi)$ satisfying,
$$
h_\delta^{\calK}(\xi)= h(\psi_\delta(\xi)) + \frac{1}{\delta} |\psi_\delta(\xi)- \xi|^2.
$$
Then, $c_{\calK} \ge h(\xi) \ge h_\delta^{\calK}(\xi) \ge h(\psi_\delta(\xi)) \ge - c_{\calK}$,
or equivalently, $ |h_\delta^{\calK}(\xi)| \le  c_{\calK}$.  Additionaly,
$$
|\psi_\delta(\xi)- \xi|^2 \le ( h(\xi)-h_\delta^{\calK}(\xi)) \delta \le 2   c_{\calK}\delta.
$$
Moreover, since $h(\psi_\delta(\xi)) \le h_\delta^{\calK}(\xi) \le h(\xi)$,
\begin{align*}
    \frakm_{\calK}(h,h_\delta^{\calK})  \le
\sup_{\xi \in \calK} \left[ ( h(\psi_\delta(\xi)) - h_\delta^{\calK}(\xi))^+ +  |\psi_\delta(\xi)-\xi|\right ]
 = \sup_{\xi \in \calK} |\psi_\delta(\xi)-\xi|
\le (2 c_{\calK} \delta)^{\tfrac12}.
\end{align*}
We now set $\delta_{\iota, \calK}:= \iota^2/(2c_{\calK})$ to complete the proof.
 \end{proof}}}

The next approximation result follows immediately. 
{\rr{\begin{theorem} \textnormal{\textbf{(Universal approximation of lower semicontinuous functions).}}
\label{thm:LSCUAT}
Let  $ \frakN \subseteq \frakF_c$ be a hypothesis space satisfying Assumption \ref{asm:UAP}. 
Then, for all $\iota >0$, $\calK \subset \subset \calE$ 
and $f \in \frakF$ that is  bounded in $\calK \times \calT$, 
there exists  $g\in \frakN$ such that \rr{$\frakm_{\calK} (f , g) \le \iota.$}  
\end{theorem}
\begin{proof}
Fix $\iota> 0$ and $\calK \subset \subset \calE$. 
In view of \cref{lem:infCovApprox}, 
$\frakm_{\calK} (f , f_{\delta}^{\calK})\le \iota/2$ 
for some $\delta > 0$.  Moreover, as $f_{\delta}^{\calK}(t,\xi)$ is continuous in $\xi$,  we can invoke Assumption \ref{asm:UAP} giving the existence of   $g\in \frakN  \subseteq \frakF_c$ such that 
$\lVert f_{\delta}^{\calK} - g \rVert_{L^{\infty}(\calT \times \calK)} \le \iota/2$.
Since $\frakm$ is dominated by the supremum norm,
the triangular inequality yields 
$$
\frakm_{\calK}(f,g)
\le \frakm_{\calK}(f,f_\delta^{\calK})
+\frakm_{\calK}(f_\delta^{\calK}, g)
\le \frac12 \iota + \lVert f_{\delta}^{\calK} - g \rVert_{L^{\infty}(\calT \times \calK)}\\
\le \iota.
$$
\end{proof}
Finally, we obtain the following   convergence theorem for lower semicontinuous  boundaries. 
\begin{theorem}
   Let $\frakN \subset \frakF_c$  be a family of universal approximators,
   and $v^{\diamond}$ be the value function defined in \eqref{eq:vdiamond}.  
   Under Assumption \ref{asm:absCont}, for every $\delta>0$, 
   there exists $\varepsilon=\varepsilon(\delta), \gamma=\gamma(\delta)>0$  such that 
   any function $g^{\varepsilon,\gamma}\in \frakN$ 
    with the property 
   $v(\mu^{\varepsilon}_{g^{\varepsilon,\gamma}}) \ge \sup_{g \in \frakN} \  v(\mu^{\varepsilon}_g) - \gamma$ 
   also satisfies,
   $$
 v(\tau_{g^{\varepsilon,\gamma}})\ge v^{\diamond} - \delta
 =v(\tau_{f^\diamond})-\delta
 = \sup_{f \in \frakF} v(\tau_f)-\delta.
 $$
\end{theorem}
\begin{proof} Let $\delta'>0 $ be arbitrary. In view of \cref{thm:lscV}, 
there are $\delta''>0$, $\calK \subset \subset \calE$, $\iota>0$ such that
\begin{equation}
\label{eq:contV}
  \frakm_{\calK}(f,g)\le \iota \ \Longrightarrow \ 
  |v(\tau_{f}) - v(\tau_{g})|\le C\sqrt{\rho(\iota) + \delta''}\le \delta', 
  \quad \forall f,g\in \frakF.  
\end{equation}
As the optimal boundary $f^{\diamond}\in \frakF$ is locally bounded, we apply  
\cref{thm:LSCUAT} to construct $g^{\diamond} \in \frakN$ such that 
$\frakm_{\calK}(f^{\diamond}, g^{\diamond})\le \iota$. 
We conclude  from \eqref{eq:contV}  that $v(\tau_{g^{\diamond}}) \ge
v(\tau_{f^{\diamond}})- \delta'= v^{\diamond} - \delta'$. As $g^{\diamond}$ is 
continuous, we can follow the 
Steps 2, 3, and 4 in the proof of \cref{thm:convCont}. 
\end{proof}}}

%% file: Conclusion.tex
\section{Conclusion}\label{sec:conclusion}
We analyze  hitting times of boundaries arising in optimal stopping. 
 After introducing relaxed stopping rules and establishing preliminary results, the continuity of the value functional $f\mapsto v(\tau_f)$ is shown with respect to the topology of uniform convergence and the novel relaxed $L^{\infty}$ distance. 
  The main results are later applied   to study the convergence of  the neural optimal stopping boundary method \cite{ReppenSonerTissotFB}. When the optimal stopping boundary is continuous, we can directly invoke the universal approximation theorem of neural networks and obtain a convergence theorem. When the optimal stopping boundary is strictly semicontinuous,  inf/sup convolutions are combined  with the relaxed $L^{\infty}$ distance, leading to a more general  result. 

A natural extension of the present work would be to show similar continuity and relaxation results for optimal stopping problems in continuous time. In particular, it would be worth investigating the convergence of the neural optimal stopping method for \textit{American} options, namely when $\calT = [0,T]$. 


%% file: Appendix.tex
\section{Proof of \cref{prop:relaxOS}}
\label{app:relaxedproblem}

\begin{proof}
   We write $\underline{v}^{\diamond}:=  \sup_{\mu \in \underline{\vartheta}(\calT)} v(\mu)$ and show that  $\underline{v}^{\diamond} = v^{\diamond}$ by induction on $N = |\calT|$. If $N=1$,  then $\calT$ is a singleton, hence $\underline{\vartheta}(\calT) = \vartheta(\calT)$. We now suppose the result true for $N\ge 1$ and prove the claim for $N+1$. Let $\calT = \{t_1<\ldots < t_{N+1}\}\subset [0,T]$ and fix  $\mu = \sum_{t\in \calT} P_t \ \delta_t \in \underline{\vartheta}(\calT)$ where each $P_t:\Omega \to [0,1]$ is $\calF_t-$measurable.  Next,  write $\calT_{\! +} = \calT \setminus  \{t_1\}$ and consider the continuation value functional  
    $$v_1(\mu_{+},x) = \E^{\Q}\left[\int_{\calT_{\! +}} \varphi(t,X_{t}) \mu_{+}(dt)  \ \big | \ X_{t_1} = x\right], \quad \mu_{+}  \in \underline{\vartheta}(\calT_{\! +}). $$
    Since $|\calT_{\! +}| = N$, we have by induction that 
\begin{equation}\label{eq:induction}
        v_1(\mu_{+},x) \le v_1^{\diamond}(x) := \sup_{\tau \in \vartheta(\calT_{\! +})} \E^{\Q}[ \varphi(\tau,X_{\tau}) \ | \ X_{t_1} = x], \quad \forall \ \mu_{+}  \in \underline{\vartheta}(\calT_{\! +}). 
    \end{equation}
    Then, the tower property of conditional expectations implies that 
\begin{align}
        v(\mu) &= \E^{\Q}\left[P_{t_1} \varphi(t_1,X_{t_1}) + \int_{\calT_{\! +}} \varphi(t,X_{t}) \mu(dt)   \right]\nonumber  \\[0.5em]
        &= 
    \E^{\Q}\left[ P_{t_1} \varphi(t_1,X_{t_1}) + (1-P_{t_1})\int_{\calT_{\! +}} \varphi(t,X_{t}) \mu_{+}(dt)\right] \nonumber\\[0.5em] 
        &= 
    \E^{\Q}\left[ P_{t_1} \varphi(t_1,X_{t_1}) + (1-P_{t_1})v_1(\mu_{+},X_{t_1}) \right], \label{eq:trick}    
    \end{align} 
    with the relaxed stopping rule 
    $$\mu_{+}(\omega) = \sum_{t\in \calT_{\! +}}\frac{P_t(\omega)}{1-P_{t_1}(\omega)} \ \delta_t \ \in \  \underline{\vartheta}(\calT_{\! +}), \qquad P_{t_1} <  1.$$ Note that if $P_{t_1}(\omega) = 1$, then $\mu_{+}(\omega)$ would not appear in \eqref{eq:trick}. 
    Together with \eqref{eq:induction}, we obtain 
    \begin{align}
     v(\mu) &\le \E^{\Q}\left[ P_{t_1} \varphi(t_1,X_{t_1}) + (1-P_{t_1})v_1^{\diamond}(X_{t_1}) \right] \nonumber \\[1em]
     &=\E^{\Q}\left[ P_{t_1} (\varphi(t_1,X_{t_1}) - v_1^{\diamond}(X_{t_1})) \right] + \E^{\Q}\left[v_1^{\diamond}(X_{t_1}) \right].\label{eq:induction2}
     \end{align}
    We can now maximize the first term in \eqref{eq:induction2}  over all $\calF_{t_1}-$measurable functions $P_{t_1}:\Omega \to [0,1]$. 
  It is straightforward to see and also   classical in  optimal stopping \cite{PeskirShiryaev} 
  that the unique maximizer is given by\footnote{In financial terms, this  means that one should exercise a Bermudan contract whenever  the intrinsic value $\varphi(t_1,X_{t_1})$ exceeds the continuation value $v_1^{\diamond}(X_{t_1})$.}   
  $$P_{t_1}^{\diamond}(\omega) = \begin{cases} 1, \quad & \text{if } \varphi(t_1,X_{t_1}(\omega)) \  \ge \ v_1^{\diamond}(X_{t_1}(\omega)), \\ 
  0,\ & \text{otherwise}.\end{cases}$$ 
  In particular, it is $\{0,1\}-$valued. 
This leads to  the dynamic programming equation 
  $$v^{\diamond} = \E^{\Q}\left[ P_{t_1}^{\diamond} (\varphi(t_1,X_{t_1}) - v_1^{\diamond}(X_{t_1})) \right] + \E^{\Q}\left[v_1^{\diamond}(X_{t_1}) \right] = \E^{\Q}\left[  \varphi(t_1,X_{t_1}) \vee v_1^{\diamond}(X_{t_1}) \right] \ge v(\mu). $$
  As
   $v^{\diamond}\le \underline{v}^{\diamond}$ is evident, \eqref{eq:OSvsRelax}  follows. 
Moving to the second assertion,  the induction hypothesis implies that  every optimal relaxed stopping rule of $\sup_{\mu_{+} \in \underline{\vartheta}(\calT_{\!+})} v_1(\mu_{+},x)$ is in fact  a genuine stopping time, say $\tau_{+}$. Together with the above arguments, we gather  that the same holds true when adding back the first exercise date $t_1$. Precisely, we have $v(\tau^{\diamond}) = \underline{v}^{\diamond}$ with the stopping time 
$\tau^{\diamond} = P_{t_1}^{\diamond} t_1  + (1-P_{t_1}^{\diamond}) \tau_{+}$.   
\end{proof}